\documentclass[fleqn,leqno]{asl}
\usepackage{newtxtext}

\theoremstyle{plain}
\newtheorem{theorem}{Theorem} 
\newtheorem{lemma}{Lemma}
\newtheorem{corollary}[theorem]{Corollary}
 
\newtheorem{question}{Question}
\theoremstyle{definition}
\newtheorem{definition}{Definition}
\newtheorem{example}{Example}

\usepackage{changes}
\usepackage{tikz}
\usetikzlibrary{automata,decorations.markings,arrows,shapes}

\tikzstyle{arrowmid}[0.65]=[decoration={markings,mark=at position #1 with {\arrow{>}}}, postaction={decorate}]
\tikzstyle{every initial by arrow}=[arrowmid,-]
\tikzset{state/.style={circle,fill=red!30,draw=black,minimum size=4pt,inner sep=2pt, outer sep=1pt}}
\tikzset{accepting/.style={fill=green!30,draw=black,minimum size=4pt,inner sep=2pt,double distance=2pt,outer sep=3pt}}
 \tikzstyle{every picture}=[initial text=,auto]
\tikzset{biglabel/.style={rectangle,fill=none,inner sep=4pt}}

\tikzstyle{bdot}=[circle,fill,draw,thick,minimum size=0.8ex,inner sep=0pt]
\tikzstyle{dot}=[circle,draw,thick,minimum size=0.8ex,inner sep=0pt]

\usetikzlibrary{cd}

\usepackage[inline]{enumitem}
\setlist[itemize]{leftmargin=*, label={--}}
\setlist[enumerate]{leftmargin=*, label={\arabic*.}, itemindent=0pt, }

\usepackage{booktabs}

\newcommand{\pw}[1]{{{\raise2pt\hbox{\large$\wp$}}(#1)}}
\newcommand{\store}{\mathord{\downarrow}}
\newcommand{\jump}{\mathord{@}}

\newcommand{\cc}[2]{#1\mathbin{\smallfrown}#2}
\newcommand{\str}[1]{\mathfrak{#1}}
\newcommand{\Prop}{\text{\normalfont\textsc{prop}}}
\NewDocumentCommand{\prop}{sE{:}{p}}{\IfBooleanTF{#1}{\mathrm{#2}}{#2}}
\newcommand{\Nom}{\text{\normalfont\textsc{nom}}}
\NewDocumentCommand{\nom}{sE{:}{s}}{\IfBooleanTF{#1}{\mathrm{#2}}{#2}}
\NewDocumentCommand{\wrld}{sE{:}{m}}{\IfBooleanTF{#1}{\mathrm{#2}}{#2}}
\NewDocumentCommand{\wrlds}{sE{:}{m}}{\overline{\IfBooleanTF{#1}{\mathrm{#2}}{#2}}}
\newcommand{\WVar}{\text{\normalfont\textsc{wvar}}}
\NewDocumentCommand{\wvar}{sE{:}{x}}{\IfBooleanTF{#1}{\mathrm{#2}}{#2}}
\NewDocumentCommand{\plc}{sE{:}{w}}{\IfBooleanTF{#1}{\mathrm{#2}}{#2}}
\NewDocumentCommand{\fovar}{sE{:}{x}}{\textbf{\IfBooleanTF{#1}{\textrm{#2}}{\textit{#2}}}}
\newcommand*{\Mod}{\mathrm{Mod}_{*}}
\NewDocumentCommand{\ST}{E{_}{x}}{\mathit{ST}_{\!\fovar:#1}}
\newcommand*{\SBT}{\mathit{SBT}}
\newcommand*{\dg}{\mathsf{dg}}
\newcommand*{\htp}[2]{\mathsf{htp}^{\leq#1}_{\str{#2}}}
\newcommand*{\http}[1]{\mathsf{http}_{\str{#1}}}
\newcommand*{\fosig}{\Sigma}
\newcommand*{\hyFeat}{\mathcal{F}}
\newcommand*{\hysig}{\Sigma}
\newcommand*{\hymodels}{\Vdash}
\newcommand*{\foProp}{\Prop^{\circ}}
\newcommand*{\fohysig}{\hysig^{\circ}}
\newcommand*{\fohysigE}{\hysig^{\circ+}}
\newcommand*{\fohysigP}{\hysig^{\circ\prime}}
\newcommand*{\fohylang}{\mathcal{L}_{\fohysig}}
\newcommand*{\fohylangE}{\mathcal{L}_{\fohysigE}}
\newcommand*{\fohylangP}{\mathcal{L}_{\fohysigP}}
\newcommand*{\fomodels}{\vDash}

\newcommand{\stacked}[1]{%
  \ifmmode\def\env{array}\else\def\env{tabular}\fi%
  \begin{\env}[t]{@{}l@{}}#1\end{\env}%
}
\allowdisplaybreaks

\begin{document}

\title[Bisimulations in hybrid logic]{A modular bisimulation characterisation
  for fragments of hybrid logic}

\author[Badia]{Guillermo Badia}
\revauthor{Badia, Guillermo}
\address{University of Queensland, Australia}
\email{g.badia@uq.edu.au}

\author[G\u{a}in\u{a}]{Daniel G\u{a}in\u{a}}
\revauthor{G\u{a}in\u{a}, Daniel}
\address{Institute of Mathematics for Industry, Kyushu University, Japan} 
\email{daniel@imi.kyushu-u.ac.jp} 

\author[Knapp]{Alexander Knapp}
\revauthor{Knapp, Alexander}
\address{University of Augsburg, Germany}
\email{knapp@informatik.uni-augsburg.de} 

\author[Kowalski]{Tomasz Kowalski}
\revauthor{Kowalski, Tomasz}
\address{Jagiellonian University, Poland\\
  La Trobe University, Australia\\
  University of Queensland, Australia}
\email{tomasz.s.kowalski@uj.edu.pl}

\author[Wirsing]{Martin Wirsing}
\revauthor{Wirsing, Martin}
\address{Ludwig-Maximilians-University Munich, Germany}
\email{wirsing@lmu.de}

\begin{abstract}
There are known characterisations of several fragments of hybrid logic by means
of invariance under bisimulations of some kind. The fragments include
$\{\store, \jump\}$ with or without nominals (Areces, Blackburn, Marx), $\jump$ with
or without nominals (ten Cate), and $\store$ without nominals (Hodkinson,
Tahiri). Some pairs of these characterisations, however, are incompatible with
one another.  For other fragments of hybrid logic no such characterisations were
known so far.  We prove a generic bisimulation characterisation theorem for all
standard fragments of hybrid logic, in particular for the case with $\store$
and nominals, left open by Hodkinson and Tahiri.
Our characterisation is built on a common
base and for each feature extension adds a specific condition, so it is
modular in an engineering sense.
\end{abstract}





\maketitle

\section{Introduction}

Hybrid logic dates back to Arthur Prior's works from 1960, but the story is
somewhat convoluted so instead of giving direct references we direct the reader
to Blackburn~\cite{Bla06} for details.  Hybrid logics can be described as a
version of modal logics with an additional machinery to refer to individual
evaluation points.  The ability to refer to specific states has several
  advantages from the point of views of logic and formal specification.  For
  example, hybrid logics allow a more uniform proof
  theory~\cite{brau11,gai-godel} and model theory~\cite{GainaBK22,GainaBK23}
  than non-hybrid modal logics. From a computer science perspective, hybrid
logics are eminently applicable to describing behavioural dynamics: \emph{if at
  this state something holds, then at some state accessible from it something
  else holds}.  This view has been particularly leveraged to the specification
and modelling of reactive and event/data-based
systems~\cite{Madeira_et_al:TCS:2018,%
  Hennicker_Knapp_Madeira:FAC:2021,%
  Rosenberger_Knapp_Roggenbach:FASE:2022}.  Other applications include reasoning
on semi-structured data~\cite{FdR06} and description logics beyond
terminological boxes~\cite{Blackburn_Tzakova:AMAI:1998}.

For definitions and an introductory treatment of modal and hybrid logics we
refer the reader to Blackburn~\emph{et al.}~\cite{BdRV01}; our terminology and
notation comes largely from there, in particular, we use $\downarrow$ and
$\jump$ as the current-state binder and the state-relativisation operator. In
general, we assume familiarity with modal and hybrid logics at the level
of~\cite{BdRV01}, anything beyond that we will define.

Just as modal logic, hybrid logic can also be viewed as a certain special
fragment of monadic second-order logic, with a number of standard tricks to get
around the need for topology. Indeed, for basic hybrid logic which will be our
only concern here, one can dispense with second order altogether and treat
second-order variables as first-order predicates. Again as in modal logic, this
\emph{standard translation} of a hybrid language to a first-order language gives
an interpretation of the former in the latter.

A celebrated characterisation theorem due to van Benthem (see~\cite{vB85})
states that a first-order formula is equivalent to a translation of a modal
formula if and only if it is invariant under bisimulation.  Similar
characterisations exist for certain fragments of hybrid logic: Areces \emph{et
  al.}~\cite{ABM01} give one for the fragment with $\jump$ and $\store$,
the fragment with $\jump$, was characterised by ten
Cate~\cite{ten05}, and Hodkinson and Tahiri
in~\cite{HT10} characterised the fragment with $\store$ but, importantly, without
nominals. As we will see, these characterisations are rather disparate and some
of them do not extend to richer fragments.
In particular, the method used in~\cite{HT10}
does not cover the language with nominals, and the authors pose the problem of
finding a characterisation for this fragment. They also tabulate existing
results, and ask, more generally, which fragments can be characterised by some
form of bisimulations. Here is a version of this table:

\begin{center}\vspace*{.8ex plus 10pt}
\begin{tabular}{ll}
\toprule
Hybrid features & Invariance under \\
\midrule
$\emptyset$ & Bisimulations \\
  $\{\store\}$ w/o nominals & Quasi-injective bisimulations\\
$\{\store\}$ with nominals & \emph{unknown}\\    
  $\{\jump\}$ & $\mathcal{H}(\jump)$-bisimulations\\
$\{\jump,\store\}$ & $\omega$-bisimulations\\  
$\{\exists\}$ & \emph{unknown}  \\
$\{\jump,\exists\}$ & equivalent to FOL \\  
full feature set & equivalent to FOL \\
\bottomrule
\end{tabular}\vspace*{.8ex plus 10pt}
\end{center}

In this article we answer these questions by providing a characterisation for
all (sensible) fragments of the hybrid language by means of a fine-tuned version
of $\omega$-bisimulation, a notion introduced in Areces \emph{et
  al.}~\cite{ABM01}, which provides one crucial ingredient for our results.  The
other crucial ingredient is Theorem~\ref{th:lindstrom} in
§5, which shows that first-order sentences cannot distinguish between $\omega$-bisimulation and hybrid counterpart of elementary equivalence.
Its proof is obviously inspired by Lindstr\"om's celebrated characterisation of first-order logic, but more specifically by an earlier use of the same technique by Badia~\cite{Bad16}. Let us remark that although Areces \emph{et
  al.}~\cite{ABM01} do discuss the possibility of applying their notion to other
fragments of the language, they do not state any results in this direction. 
Indeed, the proof of their characterisation theorem for the language
with $\store$, even without $\jump$, uses the condition for $\jump$ essentially,
to deal with nominals.
Our results are completely modular, in the engineering
sense that you can choose a subset $\hyFeat$ of the features of a hybrid
  language, and we give a characterisation theorem for the language fragment
for $\hyFeat$ using conditions pertaining to $\hyFeat$.

\section{Terminology and notation}

As we already mentioned, we follow the terminology of
Blackburn~\emph{et~al.}~\cite{BdRV01} for modal and hybrid logic. For general
model theory, we follow Hodges~\cite{Hod93}, at least in spirit.  We write
$\str{M} = (M,\allowbreak R^\str{M},\allowbreak f^\str{M},\allowbreak
c^\str{M},\allowbreak \dots)$ for structures, with $M$ being the universe, and
$R^\str{M}, f^\str{M}, c^\str{M}$ interpretations of a relation $R$, a function
$f$, and a constant $c$ of the appropriate signature $\fosig$. We use the
same notation $\str{N} = (N,\allowbreak R^\str{N},\allowbreak V^\str{N})$ for
Kripke structures, where $R$ is a
binary accessibility relation and $V$ a valuation. Often we will consider models
with a distinguished element; we will call them \emph{pointed} and write
$(\str{M}, \wrld:m)$ for a model $\str{M}$ with a distinguished element
$\wrld:m\in M$. The class of all pointed models of a formula $\phi(x)$ with a
free variable $x$ will be denoted by $\Mod(\phi)$. As usual, we dispense with
superscripts unless we have a good reason to use them; we also use typical
shorthand notation for sequences, namely, $\wrlds:m = (\wrld:m_1,\allowbreak
\dots,\allowbreak \wrld:m_n)$ together with $\wrlds:m(i) = \wrld:m_i$.  If
convenient, especially for comparing sequences of different lengths, we can also
view sequences of elements of $M$ as finite words over the alphabet $M$.

We deal exclusively with hybrid propositional logic, so the language is built
out of the atomic propositions in $\Prop \cup \Nom \cup \WVar$, where $\Prop$ is
the set of propositional variables, $\Nom$ the set of nominals, and $\WVar$ the
set of world variables.  Formulas are defined by the grammar
\begin{equation*}
\phi ::= \bot \mid \prop:p \mid \nom:s \mid \wvar:x \mid \neg\phi \mid \phi \vee \psi
\mid \Diamond\phi \mid \store \wvar:x \phi \mid \jump_{\plc:w}\psi \mid \exists x\phi 
\end{equation*}
where $\prop:p \in \Prop$, $\nom:s \in \Nom$, $\wvar:x\in\WVar$, and $\plc:w \in
\WVar \cup \Nom$. Even without $\exists$, hybrid logic is more expressive than
modal logic, since for example $\store \wvar:x \Diamond x$ expresses reflexivity
of the current state. With $\exists$ it is even stronger, for example $\neg
\exists \wvar:x(\wvar:x\wedge\Diamond \wvar:x)$ expresses irreflexivity of the
accessibility relation. In presence of $\exists$ the operation $\store$ is
redundant as $\store \wvar:x \phi$ is equivalent to $\exists
\wvar:x(\wvar:x\wedge \phi)$.  In presence of $\exists$ and $\jump$ the expressivity
of hybrid logic is the same as that of full first-order logic.

Extending the familiar concept from modal logic, we define the \emph{degree}
$\dg(\phi)$ of a hybrid formula $\phi$ in the usual
recursive way:
\begin{itemize}
  \item $\dg(\phi) = 0$ for $\phi\in\Prop\cup\Nom\cup\WVar$,
  \item $\dg(-\phi) = \dg(\phi)$ for $-\in\{\neg, \jump_{\plc:w}\}$,
  \item $\dg(\phi \mathbin{\vee} \psi) = \max\{\dg(\phi),\dg(\psi)\}$,
  \item $\dg(-\phi) = \dg(\phi) + 1$ for $-\in\{\Diamond, \store \wvar:x,
\exists \wvar:x\}$.
\end{itemize}
Note that the degree increases only on formulas whose \emph{standard
  translations} (to be recalled later) involve quantification.

\subsection{Semantics of hybrid logic}

A model for the hybrid language is a structure $\str{M} = \bigl(M, R, \Nom,
V\bigr)$, where $V\colon \Prop\to\pw{M}$ is a valuation fixing which
propositions hold in which worlds.  The presence of $\Nom$ in the model is
justified by the usual constants-name-themselves trick.  An \emph{assignment} is
any map $\nu\colon \WVar\to M$. For an assignment $\nu$, a world variable $x$,
and an element $\wrld:m\in M$ we define the $\wvar:x$-variant
$\nu^{\wvar:x}_{\wrld:m}$ of $\nu$ to be the map
\begin{equation*}
\nu^{\wvar:x}_{\wrld:m}(\wvar:y) =\begin{cases}
\wrld:m & \text{ if } \wvar:y = \wvar:x,\\
\nu(\wvar:y) & \text{ if } y \neq \wvar:x.
\end{cases}                       
\end{equation*}
Given a model $\str{M}$, an assignment $\nu\colon \WVar\to M$, and
an element $\wrld:m\in M$, we define inductively
\begin{itemize}
  \item $\str{M},\nu,\wrld:m \hymodels \bot$ never holds
  \item $\str{M},\nu,\wrld:m \hymodels \prop:p$ if $\prop:p \in \Prop$ and
$\wrld:m\in V(\prop:p)$
  \item $\str{M},\nu,\wrld:m \hymodels \nom:s$ if $\nom:s \in \Nom$ and $\wrld:m =
\nom:s^{\str{M}}$
  \item $\str{M},\nu,\wrld:m \hymodels \wvar:x$ if $\wvar:x \in \WVar$ and $\wrld:m
= \nu(\wvar:x)$
  \item $\str{M},\nu,\wrld:m \hymodels \phi \vee \psi$, if $\str{M},\nu,\wrld:m
\hymodels \phi$ or $\str{M},\nu,\wrld:m\hymodels\psi$
  \item $\str{M},\nu,\wrld:m \hymodels \neg\phi$ if $\str{M},\nu,\wrld:m \not\hymodels
\phi$
  \item $\str{M},\nu,\wrld:m \hymodels \Diamond\phi$ if
$\str{M},\nu,\wrld:m'\hymodels\phi$ for some $\wrld:m'$ with $\wrld:m \mathrel{R}
\wrld:m'$
  \item $\str{M},\nu,\wrld:m \hymodels \store \wvar:x\phi$ if $\wvar:x \in \WVar$
and $\str{M}, \nu^{\wvar:x}_{\wrld:m},\wrld:m \hymodels \phi$
  \item $\str{M},\nu,\wrld:m \hymodels \jump_{\nom:s}\phi$ if $\nom:s \in \Nom$ and
$\str{M},\nu,\nom:s^{\str{M}} \hymodels \phi$
  \item $\str{M},\nu,\wrld:m \hymodels \jump_{\wvar:x}\phi$ if $\wvar:x \in \WVar$ and
$\str{M},\nu,\nu^{\str{M}}(\wvar:x) \hymodels \phi$
  \item $\str{M},\nu,\wrld:m \hymodels \exists \wvar:x\phi$ if $\wvar:x \in \WVar$
and $\str{M},\nu^{\wvar:x}_{\wrld:m'},\wrld:m \hymodels \phi$ for some
$\wrld:m'\in M$
\end{itemize}
The assignment $\nu^\str{M}$ can equivalently be given as a sequence of
$\wrlds:m$ of elements of $M$; we will use this notation extensively in
§5. The operators $\store \wvar:x$ and $\exists \wvar:x$ \emph{bind} the
world variable $\wvar:x$, whereas $\jump_{\wvar:x}$ does not (to see this, note that
evaluating $\store \wvar:x \phi$ at $\wrld:m$ is independent of the assignment,
whereas $\jump_{\wvar:x}$ is not). An occurrence of a world variable $\wvar:x$ in a
formula $\phi$ is \emph{free} if it is not in the scope of an operator binding
$\wvar:x$, a world variable $x$ is free in $\phi$ if it has at least one free
occurrence in $\phi$. A \emph{hybrid sentence} is a hybrid formula with no free
world variables. For a sentence $\phi$ the assignment $\nu$ is irrelevant, so we
write $\str{M},\wrld:m \hymodels \phi$.

\subsection{Translations and fragments}
To each hybrid language there corresponds a first-order language consisting of a
binary predicate $R$, a set $\Nom$ of constants, and the set $\foProp$ of
predicates $P$, one for each $\prop:p \in \Prop$. For a hybrid signature
$\hysig = (\Prop,\allowbreak \Nom)$, let $\fohysig = (R,\allowbreak
  \foProp,\allowbreak \Nom)$ be the first-order signature corresponding to
$\hysig$, and let $\fohylang$ denote the corresponding first-order
language. Models over $\fohysig$ are structures $\str{M} = (M,\allowbreak
R,\allowbreak \foProp,\allowbreak \Nom)$, such that $(M,\allowbreak
R,\allowbreak \Nom,\allowbreak V)$ is a hybrid model.  Assume $\WVar
=\{z_i:i\in\omega\}$. We pick two new variables $\fovar:x$ and $\fovar:y$, and
define \emph{standard translations} $\ST_x(\phi)$ and $\ST_y(\phi)$ of a hybrid
formula $\phi$ by simultaneous recursion as follows, with $\prop:p\in \Prop$,
$\plc:w\in\Nom\cup\WVar$, and $z\in \WVar$: 
\begin{align*}
\ST_x(\bot) &= \bot & \ST_y(\bot) &= \bot    \\ 
\ST_x(\prop:p) &= P(\fovar:x) & \ST_y(\prop:p) &= P(\fovar:y) \\
\ST_x(\plc:w) &= \plc:w \approx \fovar:x & \ST_y(\plc:w) &= \plc:w \approx \fovar:y\\
\ST_x(\neg\phi) &= \neg\ST_x(\phi) & \ST_y(\neg\phi) &= \neg\ST_y(\phi)  \\
\ST_x(\phi\vee\psi) &= \ST_x(\phi) \vee \ST_x(\psi) &
\ST_y(\phi\vee\psi) &= \ST_y(\phi) \vee \ST_y(\psi) \\              
\ST_x(\Diamond\phi) &= \exists \fovar:y: \wvar:x \mathrel{R} \fovar:y
\wedge \ST_y(\phi) &
\ST_y(\Diamond\phi) &= \exists \fovar:x: \wvar:y \mathrel{R} \fovar:x
                      \wedge \ST_x(\phi)\\
\ST_x(\store \wvar:z \phi) &= \exists \wvar:z: \wvar:z \approx \fovar:x
                             \wedge \ST_x(\phi) &
\ST_y(\store \wvar:z \phi) &= \exists \wvar:z: \wvar:z \approx \fovar:y
\wedge \ST_y(\phi) \\                                                  
\ST_x(\jump_{\plc:w} \phi) &= \fovar:x \approx \plc:w \wedge \ST_x(\phi) &
\ST_y(\jump_{\plc:w} \phi) &= \fovar:y \approx \plc:w \wedge \ST_y(\phi) \\
\ST_x(\exists \wvar:z\phi) &= \exists \wvar:z: z \approx \fovar:y \wedge \ST_y(\phi) &
\ST_y(\exists \wvar:z\phi) &= \exists \wvar:z: z \approx \fovar:x \wedge \ST_x(\phi)  
\end{align*}
The way we define translations is slightly different from Areces~\emph{et
  al.}~\cite{ABM01}, but our definitions make explicit the quantificational aspect of
$\Diamond$, $\store$ and $\exists$, and the identity-like aspect of
$\jump$. The resulting concept is essentially the same as all other particular
forms of standard translations. Namely, we have that for each hybrid sentence
$\phi$ (no free world variables), the formula $\ST_x(\phi)$ has precisely one
variable, $\fovar:x$ or $\fovar:y$, free, and moreover
\bgroup\makeatletter\@mathmargin28pt\makeatother
\begin{equation}\tag{SE}\label{eq:SE}
\str{M},\wrld:m\hymodels\phi \iff \str{M} \fomodels \ST_x(\phi)[\wrld:m]
\end{equation}\egroup
where the right-hand side is the usual first-order satisfaction of $\ST_x(\phi)$
on evaluating $\fovar:x$ to $\wrld:m$; the same holds for $\ST_y(\phi)$ of
course. Alternatively, the right-hand side can be expressed in a signature
expanded by a single constant intepreted as $\wrld:m$, by $(\str{M}, \fovar:m)
\fomodels \ST_m(\phi)$. This would have the advantage of translating sentences to
sentences.

The full hybrid language over $\hysig$ is exactly as expressive as the
first-order language over $\fohysig$. To say more about expressivity, we
recall the \emph{standard back translation} introduced in Hodkinson and
Tahiri~\cite{HT10}.  Let $F$ be the function from $\fohylang$ to the hybrid
  language over $\hysig$ defined inductively by putting
\begin{align*}
F(\bot) &= \bot\\
F(P(\plc:w)) &= \jump_{F(\plc:w)}\prop:p\\
F(\wvar:x) &= \wvar:x\\
F(\nom:s) &= \nom:s\\
F(\neg\phi) &= \neg F(\phi)\\
F(\phi\vee\psi) &= F(\phi)\vee F(\psi)\\
F(\exists \wvar:x : \phi) &= \exists \wvar:x F(\phi)\\
F(\plc:w \mathrel{R} \plc:w') &= \jump_{F(\plc:w)}\Diamond F(\plc:w')\\
F(\plc:w \approx \plc:w') &= \jump_{F(\plc:w)}F(\plc:w')
\end{align*} 
The standard back translation of a formula $\phi(\wvar:x)$, with $\wvar:x$ its
only free variable, is the hybrid sentence $\SBT(\phi(\wvar:x)) = \exists
\wvar:x: \wvar:x\wedge F(\phi(\wvar:x))$.  It is easy to show that
\begin{equation*}
\str{M}\fomodels \phi[\wrld:m] \iff  \str{M},\wrld:m\hymodels\SBT(\phi(\wvar:x)) 
\end{equation*}
for any model $\str{M}$ and any $\wrld:m\in M$. It is clear from the definitions
that in presence of $\exists$ and $\jump$ the hybrid language captures the whole
first-order one.  Letting $\hyFeat = \{ \jump, \store, \exists, \Nom \}$
  be the hybrid language features beyond modal logic, and recalling that
  $\store$ is expressible in presence of $\exists$, we obtain five interesting
  fragments of the language over $\hyFeat$, namely, $\{ \jump \}$, $\{ \store \}$, $\{ \exists
  \}$, $\{ \jump, \store \}$ and $\{\jump, \exists\}$, each with and without
  nominals.  For each of these there is a natural question about a
characterisation of its standard translation, that is, a characterisation of the
set of precisely those first-order formulas that are equivalent to translations
of hybrid sentences from a given fragment.

An appropriate notion for such characterisations proved to be
\emph{bisimulation}, which we will now recall.  Let $\str{M} =
(M,R^\str{M},V^\str{M})$ and $\str{N} = (N,R^\str{N}, V^\str{N})$ be models. A
relation $B \subseteq M \times N$ is a bisimulation between $\str{M}$ and
$\str{N}$ if 
for all $(\wrld:m, \wrld:n) \in B$ the following conditions hold:
\begin{itemize}[format={\normalfont}, widest={(forth)}]
  \item[(prop)] $\wrld:m\in V^\str{M}(\prop:p)$ iff $n\in V^\str{N}(\prop:p)$
for all $\prop:p\in\Prop$
\item[(forth)] for all $\wrld:m' \in M$ with $(\wrld:m, \wrld:m') \in
R^{\str{M}}$ there is an $\wrld:n' \in N$ such that $(\wrld:n, \wrld:n') \in
R^{\str{N}}$ and $(\wrld:m', \wrld:n') \in B$;
  \item[(back)] for all $\wrld:n' \in N$ with $(\wrld:n, \wrld:n') \in
R^{\str{M}}$ there is an $\wrld:m' \in M$ such that $(\wrld:m, \wrld:m') \in
R^{\str{M}}$ and $(\wrld:m', \wrld:n') \in B$.
\end{itemize}
So defined, bisimulation is a similarity relation between models, hence by a
terminological quirk one says that two models related by a bisimulation are
\emph{bisimilar} (a consistent terminology would have the relation called
\emph{bisimilarity}, but while logics are typically consistent, logicians are
not).

For a purely modal language, van Benthem~\cite{vB85} proved that a first-order
formula is equivalent to the standard translation of a modal sentence if, and
only if, it is invariant under bisimulations. Adding nominals to the language
breaks this correspondence. For consider the structures $\str{M}$ and $\str{N}$
in Figure~\ref{fig:nominals}, with $B = \{(\wrld*:m_0,\wrld*:n_0),
(\wrld*:m_1,\wrld*:n_1), (\wrld*:m_2,\wrld*:n_1)\}$ indicated by dashed
lines. Taking $V^\str{M}(\prop:p) = M$ and $V^\str{N}(\prop:p) = N$ on all
$\prop:p\in\Prop$, we immediately have that $B$ is a bisimulation. Now, assuming
that for distinct nominals $\nom*:s$ and $\nom*:t$ we have $\nom*:s^\str{M} =
\wrld*:m_0$, $\nom*:t^\str{M} = \wrld*:m_1$, $\nom*:s^\str{N} = \wrld*:n_0$,
$\nom*:t^\str{N} = \wrld*:n_1$ we see that $B$ does not preserve nominals, since
$\str{N}, \wrld*:n_1 \hymodels \nom*:t$ but $\str{M},\wrld*:m_2 \not\hymodels
\nom*:t$.
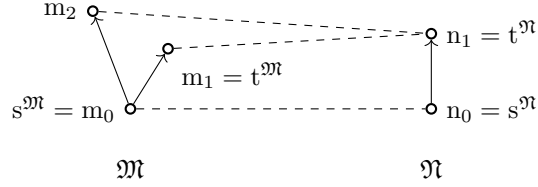
\begin{figure}\centering%
\begin{tikzpicture}
\begin{scope}
\node[dot, label=left:${\nom*:s^\str{M}=\wrld*:m_0}$] (m0) at (0, 0) {};
\node[dot, label=below right:${\wrld*:m_1=\nom*:t^\str{M}}$] (m1) at (0.5, 0.8) {};
\node[dot, label=left:${\wrld*:m_2}$] (m2) at (-0.5, 1.3) {};
\path[->]
  (m0) edge (m1)
  (m0) edge (m2)
  ;
\end{scope}
\begin{scope}[xshift=4cm]
\node[dot, label=right:${\wrld*:n_0=\nom*:s^\str{N}}$] (n0) at (0, 0) {};
\node[dot, label=right:${\wrld*:n_1=\nom*:t^\str{N}}$] (n1) at (0, 1) {};
\path[->] 
  (n0) edge (n1)
;
\end{scope}
\path[-,dashed]
(m0) edge (n0)
(m1) edge (n1)
(m2) edge (n1)
;
\node[] (M) at (0, -.8) {$\str{M}$};
\node[] (N) at (4, -.8) {$\str{N}$};
\end{tikzpicture}%
\caption{Bisimilar but not preserving nominals}\label{fig:nominals}
\end{figure}
The problem is easily repaired by strengthening the notion of bisimulation.  It
suffices to add the condition
\begin{itemize}[widest={(nom)}]
  \item[(nom)] $\wrld:m = \nom:s^\str{M}$ iff $\wrld:n = \nom:s^\str{N}$ for
every $\nom:s \in \Nom$,
\end{itemize}
to recover the characterisation. Note that (nom) forces bisimulations to be
bijections between $\Nom^\str{M}$ and $\Nom^\str{N}$.

A similar problem arises if we add $\store$ and world variables (but not
nominals) to the purely modal language.  For consider the structures $\str{M}$
and $\str{N}$ in Figure~\ref{fig:cycle}, with $B$ given by the dashed lines.
Let $V^\str{M}(\prop:p) = M$ and $V^\str{N}(\prop:p) = N$ for all
$\prop:p\in\Prop$; moreover, let $\nu^\str{M}(\wvar*:x_i) = \wrld*:m_i$ and
$\nu^\str{N}(\wvar*:x_i) = \wrld*:n_{i\,\mathbin{\mathrm{mod}}\,2}$ for all
$\wvar*:x_i \in \WVar$.
\begin{figure}\centering%
\begin{tikzpicture}
\begin{scope}
\node[dot, label=left:${\nu^\str{M}(\wvar*:x_0) = \wrld*:m_0}$] (m0) at (0, 0) {};
\node[dot, label=left:${\nu^\str{M}(\wvar*:x_1) = \wrld*:m_1}$] (m1) at (0, 1) {};
\node[dot, label=left:${\nu^\str{M}(\wvar*:x_2) = \wrld*:m_2}$] (m2) at (0, 2) {};
\node[dot, label=left:${\nu^\str{M}(\wvar*:x_3) = \wrld*:m_3}$, label=above:$\vdots$] (m3) at
(0, 3) {};
\path[->]
  (m0) edge (m1)
  (m1) edge (m2)
  (m2) edge (m3)
  ;
\end{scope}
\begin{scope}[xshift=4cm]
\node[dot, label=right:${\wrld*:n_0 = \nu^\str{M}(\wvar*:x_0) = \nu^\str{M}(\wvar*:x_2) = \dots}$] (n0) at (0, 0) {};
\node[dot, label=right:${\wrld*:n_1 = \nu^\str{M}(\wvar*:x_1) = \nu^\str{M}(\wvar*:x_3) = \dots}$] (n1) at (0, 1) {};
\path[<->] 
  (n0) edge (n1)
;
\end{scope}
\path[-,dashed]
(m0) edge (n0)
(m1) edge (n1)
(m2) edge (n0)
(m3) edge (n1);
\node[] (M) at (0, -.8) {$\str{M}$};
\node[] (N) at (4, -.8) {$\str{N}$};
\end{tikzpicture}%
\vspace*{-2ex}
\caption{Bisimilar but not preserving $\store$}\label{fig:cycle}
\end{figure}
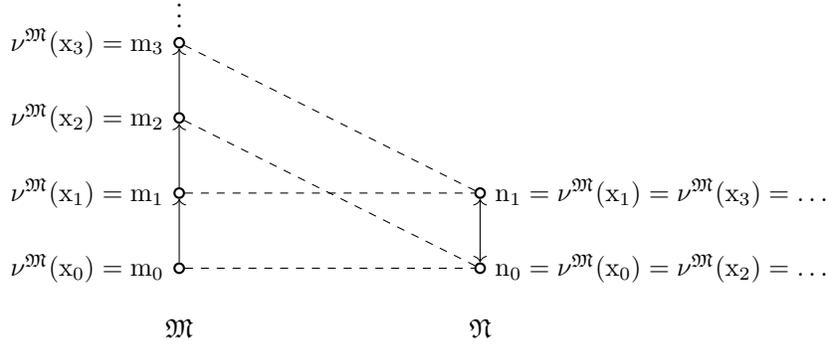
Then $B$ is a bisimulation not preserving $\store$, as for example
$\str{N},\wrld*:n_0\hymodels \store \wvar*:x \Diamond\Diamond \wvar*:x$, but
$\str{M},\wrld*:m_0\not\hymodels\store \wvar*:x \Diamond\Diamond \wvar*:x$. Similar
examples can easily be constructed with $\str{N}$ being a cycle of length $\ell$
(in particular, a loop).  As an aside, note that it would not be reasonable to
require preservation of open formulas, as nothing short of a bijection can
preserve world variables.

\section{Quasi-injective bisimulations}

To deal with $\store$, Blackburn and Seligman~\cite{BS98} introduce
\emph{quasi-injective bisimulations}, that is, relations satisfying the usual
bisimulation conditions together with the requirement that distinct bisimulation
images of the same state are mutually inaccessible.  To be precise, let $R^*$ be
the reflexive, transitive closure of $R$ and let us say that $\wrld:m'$ is
\emph{reachable} in $M$ from $\wrld:m$ if $\wrld:m' \in R^*(\wrld:m)$.  A
bisimulation $B \subseteq M \times N$ is \emph{quasi-injective} if $(\wrld:m,
\wrld:n), (\wrld:m, \wrld:n') \in B$ and $\wrld:n \neq \wrld:n'$ imply that
$\wrld:n' \notin R^*(\wrld:n)$ and $\wrld:n \notin R^*(\wrld:n')$, and the
analogous symmetric condition also holds.  They prove that formulas of the
language with $\store$ but without nominals are preserved under
quasi-injective bisimulations, hence any formula equivalent to a standard
translation of a sentence of that language is invariant under quasi-injective
bisimulations. Hodkinson and Tahiri~\cite{HT10} prove the converse, thus
obtaining a characterisation theorem.

\begin{theorem}
For a hybrid language over a signature $\hysig$ without nominals and only
involving $\store$, the following are equivalent for a first-order formula
$\phi$ over $\fohysig$:
\begin{enumerate}
  \item $\phi$ is equivalent to a translation of a hybrid sentence.
  \item $\phi$ is invariant under quasi-injective bisimulations.
\end{enumerate}  
\end{theorem}

Quasi-injective bisimulation faces an obvious problem with nominals: even
mutually inaccessible states cannot be bisimilar to a single state whenever that
single state has a name. Aware of this, Hodkinson and Tahiri ask whether a
characterisation can be obtained for the language including nominals. The
example below shows that $\store$-preserving bisimulations need not be
quasi-injective, suggesting that quasi-injectivity is too strong; see also
Example~\ref{ex:cardinality} in §4.

\begin{example}\label{ex:overkill}
Consider structures $\str{M}$ and $\str{N}$ of Figure~\ref{fig:not-q-injective}.
The relation $B$ indicated by dashed arrows is clearly a bisimulation; also
clearly it is not quasi-injective. 
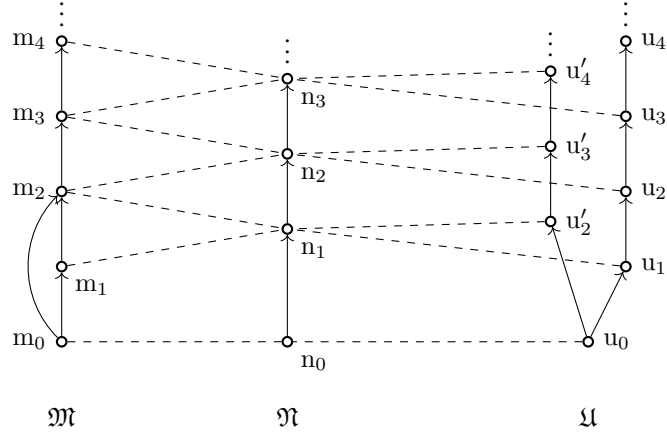
\begin{figure}
\begin{center}
\begin{tikzpicture}
\begin{scope}
\node[dot, label=left:$\wrld*:m_0$] (10) at (0, -1) {};
\node[dot, label=below right:$\wrld*:m_1$] (11) at (0, 0) {};
\node[dot, label=left:$\wrld*:m_2$] (12) at (0, 1) {};
\node[dot, label=left:$\wrld*:m_3$] (13) at (0, 2) {};
\node[dot, label=left:$\wrld*:m_4$, label=above:$\vdots$] (14) at (0, 3) {};
\path[->]
  (10) edge (11)
  (10) edge[bend left=45] (12)
  (11) edge (12)
  (12) edge (13)
  (13) edge (14)
  ;
\end{scope}
\begin{scope}[xshift=3cm]
\node[dot, label=below right:$\wrld*:n_0$] (20) at (0, -1) {};
\node[dot, label=below right:$\wrld*:n_1$] (21) at (0, 0.5) {};
\node[dot, label=below right:$\wrld*:n_2$] (22) at (0, 1.5) {};
\node[dot, label=below right:$\wrld*:n_3$, label=above:$\vdots$] (23) at (0, 2.5) {};
\path[->] 
  (20) edge (21)
  (21) edge (22)
  (22) edge (23)
;
\end{scope}
\path[-,dashed]
(10) edge (20)
(11) edge (21)
(12) edge (22)
(12) edge (21)
(13) edge (23)
(13) edge (22)
(14) edge (23);
\begin{scope}[xshift=7cm]
\node[dot, label=right:$\wrld*:u_0$] (30) at (0, -1) {};
\node[dot, label=right:$\wrld*:u_1$] (31) at (0.5, 0) {};
\node[dot, label=right:$\wrld*:u_2$] (32) at (0.5, 1) {};
\node[dot, label=right:$\wrld*:u_3$] (33) at (0.5, 2) {};
\node[dot, label=right:$\wrld*:u_4$, label=above:$\vdots$] (34)
at (0.5, 3) {};
\node[dot, label=right:$\wrld*:u'_2$] (32a) at (-0.5, 0.6) {};
\node[dot, label=right:$\wrld*:u'_3$] (33a) at (-0.5, 1.6) {};
\node[dot, label=right:$\wrld*:u'_4$, label=above:$\vdots$] (34a)
at (-0.5, 2.6) {};
\path[->]
  (30) edge (31)  (30) edge (32a)
  (31) edge (32)  (32a) edge (33a)
  (32) edge (33)  (33a) edge (34a)
  (33) edge (34)
  ;
\end{scope}
\path[-,dashed]
(30) edge (20)
(31) edge (21)
(32) edge (22)
(32a) edge (21)
(33) edge (23)
(33a) edge (22)
(34a) edge (23);
\node[] (M) at (0, -2) {$\str{M}$};
\node[] (N) at (3, -2) {$\str{N}$};
\node[] (N) at (7, -2) {$\str{U}$};
\end{tikzpicture}
\end{center}%
\vspace*{-1.5ex}%
\caption{$\str{N}$ and $\str{M}$ are $\store$-bisimilar
  but not quasi-injectively bisimilar; 
  $\str{N}$ and $\str{U}$ are quasi-injectively bisimilar.}\label{fig:not-q-injective} 
\end{figure}
Yet, $B$ preserves $\store$-formulas, as we will now show. To do so, we make
use of \emph{$\store$-unravelling}, defined in Hodkinson and
Tahiri~\cite[Defs.~3.7 and~3.8]{HT10}. For our purposes it suffices to observe
that for acyclic graphs $\store$-unravelling coincides with the usual
unravelling. Hodkinson and Tahiri show (see~\cite[Prop.~3.10]{HT10}) that a
model and its $\store$-unravelling are $\store$-bisimilar, hence
$\store$-sentences are invariant under $\store$-unravelling. Therefore,
$\store$-sentences are invariant under unravelling acyclic digraphs. Now
consider structures $\str{U}$ and $\str{N}$ of Figure~\ref{fig:not-q-injective},
and observe that (i) $\str{U}$ is an unravelling of $\str{M}$, (ii) $\str{M}$
and $\str{N}$ are acyclic digraphs, (iii) $\str{U}$ and $\str{N}$ are
quasi-injectively bisimilar. It follows that $\store$-sentences are invariant
under $B\subseteq M\times N$. But $B$ is not quasi-injective.
\end{example}

\section{Bisimulations with memory}

Our approach to bisimulations is motivated by Areces~\emph{et al.}~\cite{ABM01}
where in order to arrive at a notion of bisimulation appropriate for hybrid
languages the bisimilarity relation is endowed with memory: pairs of states are
not bisimilar \emph{per se} but in connection to their histories---we compare
strings over the alphabet of states rather than single
states. In~\cite[Sect.~3.3]{ABM01}, this idea is fleshed out in the form of $k$-
and $\omega$-bisimulations.  We recall the definitions below, with inessential
modifications in presentation and terminology.


\begin{definition}[cf.~{\cite[Sect.~3.3]{ABM01}}]\label{def:Areces-k-bisim}
Let $\str{M}$ and $\str{N}$ be Kripke structures.  
A relation $B_k \subseteq (M^k
\times M) \times (N^k \times N)$ is a \emph{$k$-bisimulation} if for all
$\bigl((\wrlds:m, \wrld:m), (\wrlds:n, \wrld:n)\bigl) \in B_k$ the following hold:
\begin{itemize}[widest={(forth)}]
  \item[(prop)] $\wrld:m \in V(\prop:p)$ iff $\wrld:n \in V'(\prop:p)$ for all
$\prop:p \in \Prop$;
  \item[(nom)] $\wrld:m = \nom:s^\str{M}$ iff $\wrld:n = \nom:s^\str{N}$ for
every $\nom:s\in\Nom$;
  \item[(wvar)] $\wrlds:m(j) = \wrld:m$ iff $\wrlds:n(j) = \wrld:n$ for all $1
\leq j \leq k$;
  \item[(forth)] for all $\wrld:m' \in M$ with $(\wrld:m, \wrld:m') \in
R^\str{M}$ there is an $\wrld:n' \in N$ with $(\wrld:n, \wrld:n') \in R^\str{N}$
and $\bigl((\wrlds:m, \wrld:m'), (\wrlds:n, \wrld:n')\bigr) \in B_k$;
  \item[(back)] for all $\wrld:n' \in N$ with $(\wrld:n, \wrld:n') \in
R^\str{N}$ there is an $\wrld:m'\in M$ with $(\wrld:m, \wrld:m') \in R^\str{M}$
and $\bigl((\wrlds:m, \wrld:m'), (\wrlds:n,\wrld:n')\bigr) \in B_k$;
  \item[(atv)] $\Bigl(\bigl(\wrlds:m, \wrlds:m(j)\bigr),
\bigl{(}\wrlds:n,\wrlds:n(j)\bigr{)}\Bigr) \in B_k$ for all $1 \leq j \leq k$;
 \item[(atn)] $\Bigl(\bigl(\wrlds:m, s^\mathfrak{M}\bigr),
\bigl{(}\wrlds:n, s^\mathfrak{N}\bigr{)}\Bigr) \in B_k$ for all $\nom:s \in \Nom$;
  \item[(bind)] $\bigl((\wrlds:m^j_{\wrld:m}, \wrld:m),
(\wrlds:n^j_{\wrld:n}, \wrld:n)\bigr) \in B_k$ for all $1 \leq j \leq k$,\\
where $(\wrld:v_1,\ldots, \wrld:v_k)^j_{\wrld:v} = (\wrld:v_1, \ldots,
\wrld:v_{j-1}, \wrld:v, \wrld:v_{j+1},\ldots, \wrld:v_k)$.
\end{itemize}
An \emph{$\omega$-bisimulation} between $\str{M}$ and $\str{N}$ is a non-empty
family of $k$-bisimulations $(B_k)_{k\in\omega}$ satisfying the following
condition for all $k \in\omega$: \bgroup\makeatletter\@mathmargin28pt\makeatother%
\begin{equation}\tag{ext}\label{eq:ext}
\bigl((\wrlds:m, \wrld:m), (\wrlds:n, \wrld:n)\bigr) \in B_k \;\implies\;   
\bigl((\cc{\wrlds:m}{\wrld:m}, \wrld:m), (\cc{\wrlds:n}{\wrld:n}, \wrld:n)\bigr) \in B_{k+1},
\end{equation}
\egroup
where $\cc{}{}$ stands for concatenation. (Note that
technically an $\omega$-bisimulation is a relation between
$M^*$ and $N^*$.)
\end{definition}
Areces~\emph{et al.}\ call (ext) the storage rule, but we change the terminology
here as the operator $\store$ is often informally called \emph{store}, one
intuitive reading of $\store \wvar:x \phi$ being ``store $\phi$ at $x$''; we
will informally call (ext) the extensibility condition.  Their ($\jump$)-rule is split
into two clauses, (atv) and (atn) handling variables and nominals separately.
For handling the existential quantification we need the following condition:
\begin{itemize} [widest={(ex)}]
 \item[(ex)] If $(\wrlds:m,\wrld:m) \mathrel{B_k} (\wrlds:m',\wrld:m')$,
then:
\begin{itemize}[widest={(ex-b)}]
  \item[(ex-f)] for all $j\leq k$, $\wrld:n\in M$, 
  
  there exists an $\wrld:n'\in
M'$ such that $(\wrlds:m^j_{\wrld:n},\wrld:m)  \mathrel{B_k}
({\wrlds:m'}^j_{\wrld:n'},\wrld:m')$,

  \item[(ex-b)] for all $j\leq k$, $\wrld:n'\in M'$, 
  
  there exists an $\wrld:n\in
M$ such that $(\wrlds:m^j_{\wrld:n},\wrld:m)  \mathrel{B_k}
({\wrlds:m'}^j_{\wrld:n'},\wrld:m')$.
\end{itemize}
\end{itemize}

\begin{definition}[$\hyFeat$-$\omega$-bisimulation] \label{def:bisimulation}
Let $\hyFeat \subseteq \{ \jump, \store, \exists, \Nom \}$ be a set of hybrid language features.
Let $\mathsf{cond}(\hyFeat)$ be the smallest subset of the set of conditions
given above, satisfying the following requirements:
\begin{itemize}
  \item if $\Nom\in\hyFeat$ then $\text{(nom)}\in \mathsf{cond}(\hyFeat)$,
\item if $\store \in\hyFeat$ then $\text{(bind)}\in
  \mathsf{cond}(\hyFeat)$, 
    \item if $\{\store,\Nom \}\subseteq\hyFeat$ then
  $\{\text{(bind)}, \text{(nom)}\}\subseteq \mathsf{cond}(\hyFeat)$,
\item if $\jump\in\hyFeat$ then
  $\text{(atv)}\in \mathsf{cond}(\hyFeat)$, 
    \item if $\{\jump,\Nom \}\subseteq\hyFeat$ then
  $\{\text{(atv)}, \text{(atn)}, \text{(nom)}\}\subseteq \mathsf{cond}(\hyFeat)$,
\item if $\exists\in\hyFeat$ then $\text{(ex)} \in \mathsf{cond}(\hyFeat)$,
\item if $\{\exists,\Nom\}\subseteq\hyFeat$ then
  $\{\text{(ex)}, \text{(nom)}\}\subseteq \mathsf{cond}(\hyFeat)$,
\item if $\{\store, \jump\}\subseteq\hyFeat$ then
  $\{\text{(bind)}, \text{(atv)}\} \subseteq \mathsf{cond}(\hyFeat)$,
    \item if $\{\store, \jump,\Nom \}\subseteq\hyFeat$ then
  $\{\text{(bind)}, \text{(atv)}, \text{(atn)}, \text{(nom)}\}\subseteq \mathsf{cond}(\hyFeat)$,
\item if $\{\exists, \jump\}\subseteq\hyFeat$ then
  $\{\text{(ex)}, \text{(atv)}\}\subseteq \mathsf{cond}(\hyFeat)$,
    \item if $\{\exists, \jump,\Nom \}\subseteq\hyFeat$ then
$\mathsf{cond}(\hyFeat)$ contains all the conditions.  
\end{itemize}
An $\omega$-bisimulation $(B_k)_{k\in\omega}$ between two Kripke structures $\str{M}$ and $\str{N}$ in the hybrid language of $\hyFeat$ will be called an \emph{$\hyFeat$-$\omega$-bisimulation} from $\str{M}$ to $\str{N}$ if the conditions $\mathsf{cond}(\hyFeat)$ hold.
\end{definition}

Note that the possibilities above are exhaustive, as $\store$ is redundant in
presence of $\exists$. In practice we will avoid using precise names such as
$\{\store,\Nom\}$-$\omega$-bisimulation, and rely on context and
circumlocutions to clarify what kind of an $\omega$-bisimulation we need. The
remainder of this paper should be read with this principle in mind.

\begin{example}\label{ex:S-family}
Consider again the structures $\str{M}$ and $\str{N}$ of
Example~\ref{ex:overkill} and the bisimulation relation $B$ between $\str{M}$
and $\str{N}$.  For each of the pairs $(\wrld*:m_0, \wrld*:n_0)$, $(\wrld*:m_i,
\wrld*:n_i)$, and $(\wrld*:m_{i+1}, \wrld*:n_i)$ in $B$ with $i \geq 1$, writing
$\lambda$ for the empty sequence, define inductively for $k \geq 0$
\begin{align*}
S^{0, 0}_0 &= S^{i, i}_0 = S_0^{i+1, i} = \{ (\lambda, \lambda) \}\\
S^{0, 0}_{k+1} &= \{ (\cc{\wrlds:m}{\wrld*:m_0}, \cc{\wrlds:n}{\wrld*:n_0})
\mid (\wrlds:m, \wrlds:n) \in S^{0, 0}_k \} \\
\begin{split}
  S^{i, i}_{k+1} &= \{ (\cc{\wrlds:m}{\wrld*:m_0}, \cc{\wrlds:n}{\wrld*:n_0}) \mid
 (\wrlds:m, \wrlds:n) \in S^{i, i}_k \}\\
  &\quad\cup
  \{ (\cc{\wrlds:m}{\wrld*:m_j}, \cc{\wrlds:n}{\wrld*:n_j}) \mid 1 \leq j \leq
    i,\ (\wrlds:m, \wrlds:n) \in S^{i, i}_k \}
\end{split}\\
\begin{split}  
S^{i+1, i}_{k+1} &= 
  \{ (\cc{\wrlds:m}{\wrld*:m_0}, \cc{\wrlds:n}{\wrld*:n_0}) \mid (\wrlds:m, \wrlds:n) \in S^{i+1, i}_k \}\\ 
  &\quad\cup
  \{ (\cc{\wrlds:m}{\wrld*:m_{j+1}}, \cc{\wrlds:n}{\wrld*:n_j}) \mid 1 \leq j \leq
    i,\ (\wrlds:m, \wrlds:n) \in S^{i+1, i}_k \}
\end{split}   
\end{align*}
and finally put
\begin{equation*}
B_k = \{ (\cc{\wrlds:m}{\wrld*:m_i}, \cc{\wrlds:n}{\wrld*:n_j}) \mid (\wrld*:m_i, \wrld*:n_j) \in B,\ (\wrlds:m, \wrlds:n) \in S^{i, j}_k \}.
\end{equation*}
In particular, $B_0 = \{ (\cc{\lambda}{m}, \cc{\lambda}{n}) \mid (m, n) \in B
\}$.  Indeed, for each $k \geq 0$, the relation $B_k$ satisfies all the
requirements of a $k$-bismulation between $\str{M}$ and $\str{N}$, except
(at). The construction of $B_k$ is a good example of what we are aiming at, so
let us dwell on it for a while. One can view the sets $S^{i,j}_k$ as recording
$k$ pairs of states visited sequentially in a ``run'' starting at $(\wrld*:m_0,
\wrld*:n_0)$ and visiting pairs accessible from $(\wrld*:m_0, \wrld*:n_0)$. In
each run, only single pairs of accessible states are recorded: for instance, if
$(\wrld*:m_1,\wrld*:n_1)$ is visited, then $(\wrld*:m_2, \wrld*:n_1)$ is not. An
analogy to think of is a bisimulation between nondeterministic automata. Less
metaphorically, the construction is such that the sequence of recorded pairs of
states $S^{i, j}_k$ can only contain pairs of states
$(\wrld*:m_k,\wrld*:n_\ell)$ such that $k\leq i$ and $\ell\leq j$, where
$(\wrld*:m_i, \wrld*:n_j)$ is the current state. Indeed, $S^{i, j}_k$ contains
all of such pairs in order to satisfy (bind).  Crucially for (wvar), the
construction separates between runs including and excluding $\wrld*:m_1$; if
$(\wrld*:m_1,\wrld*:n_1)$ has been visited, then only the pairs $(\wrld*:m_i,
\wrld*:n_i)$ are visited; if not, then only the pairs $(\wrld*:m_{i+1},
\wrld*:n_i)$. This property is preserved by transitions in $\str{M}$ and
$\str{N}$, respectively, so (forth) and (back) hold. Finally, it is easy to
verify that the family $(B_k)_{0 \leq k}$ also satisfies (ext).
\end{example}

\begin{lemma}\label{lem:q-injective-implies-omega}
Let $\str{M}$ and $\str{N}$ be $\hysig$-structures
related by a quasi-injective bisimulation.  Then there is
an $\{ \store \}$-$\omega$-bi\-sim\-u\-la\-tion between $\str{M}$ and $\str{N}$. 
\end{lemma}
\begin{proof}
Let $B$ be a quasi-injective bisimulation between $\str{M}$ and $\str{N}$.
Define $B_k \subseteq (M^k \times M) \times (N^k \times N)$ for all $k \geq 0$
as follows
\begin{gather*}
B_k = \Bigl\{
\bigl((\wrlds:m, \wrld:m), (\wrlds:n,\wrld:n)\bigr) \mid (\wrld:m, \wrld:n) \in B \wedge{}\\
\hspace*{6em}\forall j\in\{1,\dots,k\} : (\wrlds:m(j), \wrlds:n(j)) \in B \wedge{}\\
\hspace*{10em}\wrld:m \in (R^\str{M})^*(\wrlds:m(j)) \wedge \wrld:n \in (R^\str{N})^*(\wrlds:n(j))
\Bigr\}
\end{gather*}
Properties (back), (forth), (bind) and (ext) are satisfied immediately.
For~(wvar), let $\bigl((\wrlds:m, \wrld:m), (\wrlds:n, \wrld:n)\bigr)\in B_k$.
First assume $\wrlds:m(j) = \wrld:m$ for some $1 \leq j \leq k$.  Then
$(\wrlds:m(j),\allowbreak \wrlds:n(j)) \in B$, $(\wrld:m,\allowbreak \wrld:n)
\in B$, and $\wrld:n\in (R^\str{N})^*(\wrlds:n(j))$ by definition of $B_k$.  But
then $\wrlds:n(j) = \wrld:n$ as otherwise $B$ would not be quasi-injective.  For
the other direction just interchange $\wrld:n$ and $\wrld:m$.
\end{proof}

As demonstrated by Example~\ref{ex:overkill}, the converse of Lemma~\ref{lem:q-injective-implies-omega} does not hold. Indeed,
the next example shows that quasi-injective bisimulations tend to preserve
non-first-order properties such as having uncountable cardinality. It would 
be interesting to find out the precise strength of quasi-injective
bisimulations, but it is beyond the scope of the present article, and may prove
elusive.  

\begin{example}\label{ex:cardinality}
Let $\str{R} = (\mathbb{R}, {\leq})$ and $\str{Q} = (\mathbb{Q}, {\leq})$ be the
usual reals and rationals with their usual orders. The player $\exists$, say,
Elo\"ise, has a winning strategy in the Ehrenfeucht-Fra\"iss\'e game
$\mathrm{EF}_\omega(\str{R},\str{Q})$. Let the relation $B_k\subseteq
(\mathbb{R}^k\times \mathbb{R})\times (\mathbb{Q}^k\times\mathbb{Q})$ be given
by putting $(\wrlds:r,\wrld:r) B_k (\wrlds:q,\wrld:q)$ if and only if (i)
$(\wrlds:r,\wrlds:q)$ is the pair of sequences arising from the first $k$ rounds
in some play of $\mathrm{EF}_\omega(\str{R},\str{Q})$ in which Elo\"ise follows
her winning strategy, and (ii) $(\wrld:r,\wrld:q)$ is the pair consisting of
Abelard's move and Elo\"ise's response in the $k+1$ round.  It is not difficult
to see that the family $(B_k)_{k\in\omega}$ is an $\omega$-bisimulation between
$\str{R}$ and $\str{Q}$. However, there is no non-empty quasi-injective
bisimulation between $\str{R}$ and $\str{Q}$. For suppose there is one, say $S$,
and suppose $r \mathrel{S} q$. Then, by quasi-injectivity, $S$ must be a
bijective map between $\{u\in\mathbb{R}: \wrld:r\leq u\}$ and $\{w\in\mathbb{Q}:
\wrld:q\leq w\}$. This is clearly impossible because of the cardinalities of
these sets.
\end{example}

\section{Characterisation theorems}

Our proof approach of bisimulation invariance characterisations for hybrid logic
is modelled after Badia~\cite{Bad16}, where such a characterisation was given for
bi-intuitionistic logic. That in turn was motivated by the proof of the
celebrated Lindstr\"om characterisation of first-order logic in~\cite{Lin69}.
The main technicality in our approach is a finite approximation of the notion of 
$k$-bisimulation in the sense of Areces~\emph{et al.}~\cite{ABM01}.
An illustration is the family $S^{i,j}_k$ of relations,
given in Example~\ref{ex:S-family}. In general, this allows us to manoeuvre
around first-order undefinability of $\omega$-bisimulations.

\begin{definition}[Basic $(k,\ell)$-bisimulation]\label{def:base-bi} Let 
$\str{M} = (M, R, V)$ and $\str{M}' = (M',\allowbreak R',\allowbreak V')$ be models of the
purely modal language.  A system of relations $(Z^k_{i})_{i \leq \ell}$ where
$Z_i^k \subseteq (M^k \times M) \times ({M'}^k \times M')$ will be called a
\emph{basic $(k,\ell)$-bisimulation} from $\str{M}$ to $\str{M}'$ if the
following holds:
\begin{itemize}[widest={(forth)}]
  \item[(prop)] if $(\wrlds:m, \wrld:m) \mathrel{Z^k_{i}} (\wrlds:m',
\wrld:m')$, then $m \in V(\prop:p)$ iff $m'\in V'(\prop:p)$ for all $\prop:p \in
\Prop$,

  \item[(wvar)] if $(\wrlds:m, \wrld:m) \mathrel{Z_i^k} (\wrlds:m', \wrld:m')$,
then $\wrlds:m(j)= \wrld:m$ iff $\wrlds:m'(j) = \wrld:m'$ for all $j\leq k$,
  
  \item[(forth)] if $(\wrlds:m, \wrld:m) \mathrel{Z_{i-1}^k} (\wrlds:m',
\wrld:m')$ and $\wrld:m \mathrel{R} \wrld:n$, then there is an $\wrld:n'$ with
$\wrld:m' \mathrel{R'} \wrld:n'$ and $(\wrlds:m, \wrld:n) \mathrel{Z_i^k}
(\wrlds:m',\wrld:n')$,

  \item[(back)] if $(\wrlds:m,\wrld:m) \mathrel{Z_{i-1}^k} (\wrlds:m',\wrld:m')$
and $\wrld:m' \mathrel{R'} \wrld:n'$, then there is an $\wrld:n$ with $\wrld:m
\mathrel{R} \wrld:n$ and $(\wrlds:m, \wrld:n) \mathrel{Z_i^k}
(\wrlds:m',\wrld:n')$.
\end{itemize}
For pointed models $(\str{M},\wrld:m)$ and $(\str{M}',\wrld:m')$, we add the
condition $\wrld:m \mathrel{Z_0^0} \wrld:m'$.
\end{definition}

\begin{definition}[Extended $(k,\ell)$-bisimulation]\label{ext-bi}
Let $\hysig$ be a hybrid signature and let $\str{M}= (M, R, \Nom, V)$ and
$\str{M}'= (M', R', \Nom, V')$ be Kripke structures over $\hysig$.  Let
$(Z^k_{i})_{i \leq \ell}$ be a basic $(k,\ell)$-bisimulation from $\str{M}$ to
$\str{M}'$.  Consider the conditions below.
\begin{itemize}[widest={(bind)}]
  \item[(nom)] If $(\wrlds:m, \wrld:m) \mathrel{Z^k_{i}} (\wrlds:m', \wrld:m')$,
then $\wrld:m = \nom:s^\str{M}$ iff $\wrld:m' = \nom:s^\str{N}$ for every
$\nom:s \in \Nom$.

  \item[(bind)] If $(\wrlds:m, \wrld:m) \mathrel{Z_{i-1}^k}(\wrlds:m',
\wrld:m')$, then for all $j\leq k$, $(\wrlds:m^j_m, \wrld:m) \mathrel{Z_i^k}
({\wrlds:m'}^j_{\wrld:m'}, \wrld:m')$.

  \item[(atv)] $\bigl(\wrlds:m, \wrlds:m(j)\bigr) \mathrel{Z_i^k}
\bigl(\wrlds:m',\wrlds:m'(j)\bigr)$ for all $1 \leq j \leq k$.

  \item[(atn)] $(\wrlds:m,\nom:s^\str{M}) \mathrel{Z^k_{i}}
(\wrlds:m',\nom:s^{\str{M}'})$ for all $\nom:s \in \Nom$.

  \item[(ex)] If $(\wrlds:m,\wrld:m) \mathrel{Z_{i-1}^k} (\wrlds:m',\wrld:m')$,
then:
\begin{itemize}[widest={(ext-b)}]
  \item[(ex-f)] for all $j\leq k$, $\wrld:n\in M$, there exists an $\wrld:n'\in
M'$ such that $(\wrlds:m^j_{\wrld:n},\wrld:m) \mathrel{Z_i^k}
({\wrlds:m'}^j_{\wrld:n'},\wrld:m')$,
  \item[(ex-b)] for all $j\leq k$, $\wrld:n'\in M'$, there exists an $\wrld:n\in
M$ such that $(\wrlds:m^j_{\wrld:n},\wrld:m) \mathrel{Z_i^k}
({\wrlds:m'}^j_{\wrld:n'},\wrld:m')$.
\end{itemize}
\end{itemize}
For a set $\hyFeat$ of hybrid language features, the basic
$(k,\ell)$-bisimulation $(Z^k_{i})_{i \leq \ell}$ will be called an
\emph{$\hyFeat$-$(k,\ell)$-bisimulation} from $\str{M}$ to $\str{M}'$ if the
conditions $\mathsf{cond}(\hyFeat)$ from Definition~\ref{def:bisimulation} hold.
For pointed models $(\str{M},\wrld:m)$ and $(\str{M}',\wrld:m')$, we add the condition
$\wrld:m \mathrel{Z_0^0} \wrld:m'$, as before. 
\end{definition}

Note that $\emptyset$-$(0,\ell)$-bisimulation is the bisimulation for the purely modal
language expanded by world variables; removing world variables gives precisely
the original bisimulation from modal logic.  At the other extreme, existence of
$\{ \jump, \store, \exists \}$-$(k,\ell)$-bi\-sim\-u\-la\-tions for every $k\in\omega$ and $\ell\in\omega$
is strictly stronger than elementary equivalence of $\str{M}$ and $\str{M}'$; we
will come back to that in §\ref{sec:final}.  

\begin{lemma}\label{karp1}
Consider a hybrid language defined over a finite signature with features from $\hyFeat$ and let $\str{M}$ and $\str{M}'$ be models for
  this hybrid language.  Let $\ell,k$ be two natural numbers.
Then the following are equivalent:
\begin{enumerate}[label={(\roman*)}, leftmargin=*, widest={ii}]
  \item For each hybrid formula $\phi$ of degree at most $\ell$ with world
variables from a finite set $\{\wvar:x_1, \dots, \wvar:x_k\}$, we have $\str{M},
\wrlds:m, \wrld:m \hymodels \phi$ iff $\str{M}', \wrlds:m', \wrld:m' \hymodels \phi$.

  \item There exists an $\hyFeat$-$(k,\ell)$-bisimulation
$(Z^k_{i})_{\substack{i \leq \ell}}$ from $\str{M}$ to $\str{M}'$ such that
\begin{enumerate} 
\item $(\wrlds:m,\allowbreak \wrld:m) \mathrel{Z^k_{0}} (\wrlds:m',\allowbreak
\wrld:m')$, and
\item ${Z^k_{0}} \subseteq \dots \subseteq Z^k_{i} \subseteq \dots \subseteq Z^k_{\ell}$.
\end{enumerate}
\end{enumerate}
\end{lemma}
\begin{proof}
For the implication from (ii) to (i), we first show that the following holds:
\begin{itemize}[widest={($\ast$)}]
  \item[($\ast$)] if $(\wrlds:m, \wrld:m) \mathrel{Z^k_{{\ell}-i}} (\wrlds:m', \wrld:m')$
then $\str{M}, \wrlds:m, \wrld:m \hymodels \phi$ iff $\str{M}',
\wrlds:m', \wrld:m' \hymodels \phi$\\
for all formulas $\phi(\wvar:x_1, \dots, \wvar:x_k)$ of degree $i \leq \ell$.
\end{itemize}
We proceed by induction on complexity of $\phi$. 

\newcommand*{\justequiv}[1][]{\quad\llap{${}\stackrel{\text{#1}}{\iff}{}$}{}}
\newcommand*{\justimpl}[1][\smash{\phantom{abcdef}}]{\quad\llap{${}\stackrel{\text{#1}}{\;\Longrightarrow\;}{}$}{}}

\smallskip\noindent%
\emph{Case $\prop:p \in \Prop$}: Since $\prop:p$ is of degree $0$, we assume $(\wrlds:m,
\wrld:m) \mathrel{Z^k_{\ell}} (\wrlds:m', \wrld:m')$. Then we have
\begin{align*}&
\str{M}, \wrlds:m, \wrld:m \hymodels \prop:p
\\\justequiv&
\wrld:m \in V(\prop:p)
\\\justequiv[(prop)]&
\wrld:m' \in V'(\prop:p)
\\\justequiv&
\str{M}', \wrlds:m', \wrld:m' \hymodels \prop:p
\end{align*}

\noindent%
\emph{Case $\nom:s \in \Nom$}: As previously, assume
$(\wrlds:m, \wrld:m) \mathrel{Z^k_{\ell}} (\wrlds:m', \wrld:m')$; then
\begin{align*}&
\str{M}, \wrlds:m, \wrld:m \hymodels \nom:s
\\\justequiv&
\wrld:m = \nom:s^\str{M}
\\\justequiv[(nom)]&
\wrld:m' = \nom:s^{\str{M}'}
\\\justequiv&
\str{M}', \wrlds:m', \wrld:m' \hymodels \nom:s
\end{align*}

\noindent%
\emph{Case $\wvar:x_e$}: Similarly, using (wvar) we obtain
\begin{align*}&
\str{M}, \wrlds:m,m \hymodels \wvar:x_e
\\\justequiv&
\wrlds:m(e)=\wrld:m
\\\justequiv&
\wrlds:m'(e) =\wrld:m'
\\\justequiv&
\str{M'}, \wrlds:m',\wrld:m' \hymodels \wvar:x_e
\end{align*}

\noindent%
The \emph{cases} of $\neg$, $\wedge$, and $\vee$ are straightforward consequences of
the induction hypothesis.

\smallskip\noindent%
\emph{Case $\phi = \Diamond\psi$}: We have $\dg(\phi)=i\leq\ell$ and $\dg(\psi)
= i-1$.  Assume that $(\wrlds:m,\wrld:m) \mathrel{Z^k_{\ell-i}}
(\wrlds:m',\wrld:m')$.
Hence
\begin{align*}&
\str{M}, \wrlds:m, \wrld:m \hymodels \Diamond \psi
\\\justequiv&
\str{M}, \wrlds:m, \wrld:n \hymodels \psi
\text{ for some } \wrld:n \text{ with } \wrld:m \mathrel{R^\str{M}} \wrld:n
\\\justimpl[I.~H.]&
\str{M'}, \wrlds:m', \wrld:n' \hymodels \psi
\text{ for some } \wrld:n' \text{ with } \wrld:m' \mathrel{R^{\str{M}'}} \wrld:n' 
\\\justequiv&
\str{M'}, \wrlds:m', \wrld:m' \hymodels \Diamond\psi
\end{align*}
where the induction hypothesis is applicable since
$(\wrlds:m,\wrld:m) \mathrel{Z^k_{\ell-i}} (\wrlds:m', \wrld:m')$ implies 
$(\wrlds:m,\wrld:n) \mathrel{Z^k_{\ell-(i-1)}} (\wrlds:m', \wrld:n')$ by (forth); 
the backward direction follows similarly using (back).

\smallskip\noindent%
\emph{Case $\phi = \store{x}_e \psi$}: Again, 
$\dg(\phi)=i\leq\ell$ and $\dg(\psi) = i-1$.
\begin{align*}&
\str{M}, \wrlds:m, \wrld:m \hymodels \store{x}_e \psi
\\\justequiv&
\str{M}, \wrlds:m^e_{\wrld:m}, \wrld:m \hymodels \psi
\\\justequiv[I.~H.]& 
\str{M'}, {\wrlds:m'}^e_{\wrld:m'}, \wrld:m' \hymodels \psi
\\ \justequiv&
\str{M'}, \wrlds:m', \wrld:m' \hymodels  \store{x}_e \psi
\end{align*}
where I.~H.\ is applicable since by (bind), we have that $(\wrlds:m,\wrld:m) \mathrel{Z^k_{{\ell-i}}}
(\wrlds:m', \wrld:m')$ implies $(\wrlds:m,\wrlds:m(e)) \mathrel{Z^k_{{\ell-(i-1)}}}
(\wrlds:m', \wrlds:m'(e))$.

\smallskip\noindent%
\emph{Case $\phi = \jump_{\wvar:x_e}\psi$}: 
\begin{align*}&
\str{M}, \wrlds:m, \wrld:m \hymodels \jump_{\wvar:x_e}\psi
\\\justequiv&
\str{M}, \wrlds:m, \wrlds:m(e) \hymodels  \psi
\\\justequiv[I.~H.]&
\str{M'}, \wrlds:m',\wrlds:m'(e) \hymodels  \psi
\\\justequiv&
\str{M'}, \wrlds:m', \wrld:m' \hymodels \jump_{\wvar:x_e}\psi
\end{align*}
where I.~H.\ is applicable since 
$(\wrlds:m,\wrld:m) \mathrel{Z^k_{{\ell-i}}} (\wrlds:m', \wrld:m')$ implies 
$(\wrlds:m,\wrlds:m(e)) \mathrel{Z^k_{{\ell-i}}} (\wrlds:m', \wrlds:m'(e))$ by (atv).

\smallskip\noindent%
\emph{Case $\phi = \jump_{\nom:s}\psi$}: 
\begin{align*}&
\str{M}, \wrlds:m, \wrld:m \hymodels  \jump_{\nom:s}\psi
\\\justequiv&
\str{M}, \wrlds:m, \nom:s^{\str{M}} \hymodels \psi
\\\justequiv[I.~H.]&
\str{M'}, \wrlds:m', \nom:s^{\str{M}'} \hymodels \psi
\\\justequiv&
\str{M'}, \wrlds:m', \wrld:m' \hymodels \jump_{\nom:s} \psi
\end{align*}
where I.~H.\ is applicable since 
$(\wrlds:m,\wrld:m) \mathrel{Z^k_{{\ell-i}}} (\wrlds:m', \wrld:m')$ implies 
$(\wrlds:m,s^\str{M}) \mathrel{Z^k_{{\ell-i}}} (\wrlds:m', s^\str{M})$ by (atn).

\smallskip\noindent%
\emph{Case $\phi = \exists \wvar:x_e\psi$}: 
We have $\dg(\phi)=i\leq\ell$ and $\dg(\psi) = i-1$.
\begin{align*}&
\str{M}, \wrlds:m, \wrld:m \hymodels \exists \wvar:x_e \psi
\\\justequiv&
\str{M}, \wrlds:m^e_{\wrld:n}, \wrld:m \hymodels \psi
\text{ for some } \wrld:n \in M
\\\justimpl[I.~H.]&
\str{M'}, {\wrlds:m'}^e_{\wrld:n'}, \wrld:m' \hymodels \psi
\text{ for some } \wrld:n' \in M'
\\\justequiv&
\str{M'}, \wrlds:m', \wrld:m' \hymodels \exists \wvar:x_e \psi
\end{align*}
where I.~H.\ is applicable since by (ext-f) we have that
$(\wrlds:m,\wrld:m) \mathrel{Z^k_{\ell-i}} (\wrlds:m', \wrld:m')$ implies 
$(\wrlds:m^e_{\wrld:n},\wrld:m) \mathrel{Z_{{\ell-(i-1)}}^k} ({\wrlds:m'}^e_{\wrld:n'},\wrld:m')$; the
converse direction follows similarly using (ext-b). 
This ends the proof of statement ($\ast$).

By statement ($\ast$), $(\wrlds:m, \wrld:m) \mathrel{Z^k_{{0}}} (\wrlds:m',
\wrld:m')$, implies that $\str{M}, \wrlds:m,\wrld:m \hymodels \phi$ iff $\str{M}',
\wrlds:m',\wrld:m' \hymodels \phi$ for all formulas $\phi[x_1,\dots,x_k]$ of degree
$\ell$.  To obtain the same for smaller degrees, given a formula
$\psi[x_1,\dots,x_k]$ of degree $i\leq\ell$ choose a tautology
$\phi[x_1,\dots,x_k]$ of degree $\ell$; then we have $\dg(\psi\wedge
\phi)=\ell$, and by statement ($\ast$) we get ($\str{M}, \wrlds:m,\wrld:m
\hymodels \psi\wedge \phi$ iff $\str{M}', \wrlds:m',\wrld:m' \hymodels \psi \wedge
\phi$); since $\phi$ is a tautology, we get ($\str{M}, \wrlds:m,\wrld:m \hymodels
\psi$ iff $\str{M}', \wrlds:m',\wrld:m' \hymodels \psi$), as required.  This
completes the proof of the implication from (ii) to (i).

\smallskip

For the implication from (i) to (ii), we show that if $\str{M}, \wrlds:m,\wrld:m
\hymodels \phi$ iff $\str{M}', \wrlds:m', \wrld:m' \hymodels \phi$ for all formulas
$\phi(x_1,\dots,x_k)$ of degree $i$, then $(\wrlds:m,\wrld:m) \mathrel{Z^k_{\ell-i}}
(\wrlds:m',\wrld:m')$.

Take a model $\str{N}$ and a tuple of points
$(\wrlds:n, \wrld:n)$ of length $k$ from $\str{N}$. Let $\htp{j}{N}(\wrlds:n,
\wrld:n)$ be the set of hybrid formulas of degree $j$ that $(\wrlds:n, \wrld:n)$
satisfies. Clearly, $\htp{j}{N}(\wrlds:n, \wrld:n)$ depends on the hybrid
language features we are considering, so we will start with the basic, empty set of
features, and then we deal with the strengthenings.  Note that by closure on
Boolean negation (which holds in any hybrid language), for every hybrid formula
$\phi$ of degree at most $j$, we have either $\phi \in
\htp{j}{N}(\wrlds:n, \wrld:n)$ or $\neg\phi \in \htp{j}{N}(\wrlds:n, \wrld:n)$.

We define $(Z^k_{i})_{\substack{i \leq \ell}}$ by putting 
\begin{equation*}
(\wrlds:n,\wrld:n) \mathrel{Z^k_i} (\wrlds:n',\wrld:n') \text{ if, and only if, }
\htp{\ell-i}{M}(\wrlds:n, \wrld:n) = \htp{\ell-i}{M'}(\wrlds:n', \wrld:n')
\ \text{.} 
\end{equation*}
Then $(\wrlds:m, \wrld:m) \mathrel{Z^k_{0}} (\wrlds:m', m')$ is obvious, and by
definition, we have $Z^k_{i-1} \subseteq Z^k_{i}$ for all $0<i\leq \ell$.

\smallskip\noindent%
\emph{Condition (prop)}:
Since $\prop:p \in \Prop$ is of degree $0$ we have $(\wrlds:m,\wrld:m)
\mathrel{Z^k_0} (\wrlds:m',\wrld:m')$, which implies $\wrld:m \in V(\prop:p)
\iff \wrld:m' \in 
V'(\prop:p)$ for all $\prop:p \in \Prop$. Hence (prop) holds.  

\smallskip\noindent%
\emph{Condition (wvar)}:
If $(\wrlds:m,\wrld:m)
\mathrel{Z_i^k} (\wrlds:m',\wrld:m')$, then for any $j\leq k$ we have
\begin{align*}
\wrlds:m(j)=\wrld:m
&\iff
\str{M}, \wrlds:m,\wrld:m \hymodels x_j
\\&\iff
\str{M'}, \wrlds:m',\wrld:m' \hymodels x_j
\iff
\wrlds:m'(j) =\wrld:m',
\end{align*}
showing that (wvar) holds. 

\smallskip\noindent%
\emph{Condition (forth)}: 
Assume that $(\wrlds:m,\wrld:m) \mathrel{Z^k_i} (\wrlds:m',\wrld:m')$ and
let $\wrld:n\in R^\str{M}(\wrld:m)$.
As the set $\htp{{\ell-(i+1)}}{\str{M}}(\wrlds:m, \wrld:n)$ contains finitely many
formulas up to logical equivalence, we can write a formula logically equivalent
to $\bigwedge\htp{{\ell-(i+1)}}{\str{M}}(\wrlds:m, \wrld:n)$.
We have that $\Diamond \bigwedge\htp{{\ell-(i+1)}}{\str{M}}(\wrlds:m,\wrld:n)\in
\htp{{\ell-i}}{\str{M}}(\wrlds:m,\wrld:m) = \htp{{\ell-i}}{\str{M}'}(\wrlds:m',\wrld:m')$,
which implies that
$(\str{M}',\wrlds:m',\wrld:n')\hymodels
\bigwedge\htp{{\ell-(i+1)}}{\str{M}}(\wrlds:m,\wrld:n)$ for some element ${\wrld:n'} \in R^\str{M'}{(m')}$. 
Therefore, $\htp{{\ell-(i+1)}}{\str{M}}(\wrlds:m,\wrld:n) \subseteq
\htp{{\ell-(i+1)}}{\str{M'}}(\wrlds:m',\wrld:n')$. 

Suppose towards a contradiction that $\phi\in
\htp{{\ell-(i+1)}}{\str{M'}}(\wrlds:m',\wrld:n')\setminus
\htp{{\ell-(i+1)}}{\str{M}}(\wrlds:m,\wrld:n)$. 
We get $(\str{M}',\wrlds:m',\wrld:n')\hymodels \phi$ and $\neg \phi \in \htp{{\ell-(i+1)}}{\str{M}}(\wrlds:m,\wrld:n)$. 
Since  $(\str{M}',\wrlds:m',\wrld:n')\hymodels \bigwedge\htp{{\ell-(i+1)}}{\str{M}}(\wrlds:m,\wrld:n)$, 
we have that $(\str{M}',\wrlds:m',\wrld:n')\hymodels \neg\phi$, which is a contradiction.
Hence, $\htp{{\ell-(i+1)}}{\str{M}}(\wrlds:m,\wrld:n) = \htp{{\ell-(i+1)}}{\str{M'}}(\wrlds:m',\wrld:n')$, 
which means $(\wrlds:m,\wrld:n) \mathrel{Z^k_{{i+1}}} (\wrlds:m',\wrld:n')$.

This finishes the proof of the lemma for basic bisimulations. Now we consider
the extensions, one by one.

\smallskip\noindent%
\emph{Condition (nom)}: We have
$
\str{M},\wrlds:m,\wrld:m\hymodels \nom:s \iff \str{M}',\wrlds:m',\wrld:m'\hymodels \nom:s
$
by assumption, and so  
$
\wrld:m = \nom:s^\str{M} \iff \wrld:m' = \nom:s^{\str{M}'}
$
by definition of satisfaction.

\smallskip\noindent%
\emph{Condition (bind)}: Assume that $(\wrlds:m,\wrld:m) \mathrel{Z^k_i}
(\wrlds:m',\wrld:m')$.  Let $j\leq k$ and furthermore $\phi\in
\htp{{\ell-(i+1)}}{\str{M}}(\wrlds:m^j_m, m)$.
By definition, $\str{M}, \wrlds:m^j_{m}, m \hymodels \phi$, 
which is equivalent to $\str{M}, {\wrlds:m}, m \hymodels \store{x_j}\phi$. 
It follows that $\store{x_j}\phi \in \htp{{\ell-i}}{\str{M}}(\wrlds:m,\wrld:m) = \htp{{\ell-i}}{\str{M}'}(\wrlds:m',\wrld:m')$.  
By definition, we obtain $\str{M'},
\wrlds:m',\wrld:m' \hymodels \store{x_j}\phi$, which is equivalent to $\str{M'},
{\wrlds:m'}^j_{m'},m' \hymodels \phi$.  
We get $\phi\in \htp{{\ell-(i+1)}}{\str{M'}}({\wrlds:m'}^j_{m'},m')$.  
The converse inclusion is proved similarly.  
Hence, $\htp{{\ell-(i+1)}}{\str{M}}(\wrlds:m^j_m, m)=\htp{{\ell-(i+1)}}{\str{M'}}({\wrlds:m'}^j_{m'},m')$, 
which means $(\wrlds:m^j_m, m) \mathrel{Z^k_{i+1}} ({\wrlds:m'}^j_{m'},m')$.

\smallskip\noindent%
\emph{Condition (ex)}: This case is similar to (forth).

\smallskip\noindent%
\emph{Condition (atv)}: 
Again, suppose $(\wrlds:m,\wrld:m) \mathrel{Z^k_i} (\wrlds:m',\wrld:m')$;
recall that by definition this means $\htp{{\ell-i}}{M}(\wrlds:m, \wrld:m) = \htp{{\ell-i}}{M'}(\wrlds:m', \wrld:m')$.
Take $j\leq k$ and let $\phi\in\htp{{\ell-i}}{M}(\wrlds:m, \wrlds:m(j))$.
Then, we obtain
\begin{align*}
\phi\in\htp{{\ell-i}}{M}(\wrlds:m, \wrlds:m(j))
&\iff \str{M},\wrlds:m, \wrlds:m(j)\hymodels \phi\\
&\iff \str{M},\wrlds:m, m\hymodels \jump_{\wvar:x_j}\phi\\
&\iff \jump_{\wvar:x_j}\phi\in\htp{{\ell-i}}{M}(\wrlds:m, \wrld:m)\\
&\iff \jump_{\wvar:x_j}\phi\in\htp{{\ell-i}}{M'}(\wrlds:m', \wrld:m')\\ 
&\iff \phi\in\htp{{\ell-i}}{M'}(\wrlds:m', \wrlds:m'(j))
\end{align*}
showing that $\htp{{\ell-i}}{M}(\wrlds:m, \wrlds:m(j)) = \htp{{\ell-i}}{M'}(\wrlds:m', \wrlds:m'(j))$. 
By definition, we obtain $(\wrlds:m, \wrlds:m(j)) \mathrel{Z^k_i} (\wrlds:m', \wrlds:m'(j))$,  which proves (atv). 
The argument for (atn) is analogous.
\end{proof}

As we mentioned already, a $k$-bisimulation can be viewed as relation between
words of length $k+1$. Now we extend the notion to words of countably infinite
arbitrary lengths.

\begin{definition}[$\hyFeat$-($\omega,\ell$)-bisimulation]\label{def:o-l-bisim}
Let $\hysig$ be a hybrid signature and let $\hyFeat$ be a set of hybrid language features.
Let $\str{M}= (M, R, \Nom, V)$ and $\str{M}'= (M', R', \Nom, V')$ be models over $\hysig$.  
Let $(Z^k_{i})_{\substack{i \leq \ell\\ k \in \omega}}$ be a system of relations such that $(Z^k_{i})_{\substack{i \leq \ell}}$ is an $\hyFeat$-$(k,\ell)$-bisimulation from $\str{M}$ to $\str{M}'$ for each $k\in\omega$. 
If the following \emph{extensibility} condition is satisfied, for all $k\in\omega$ and all $i<\ell$,
\begin{itemize}[widest={(ext)}]
\item[(ext)] for all $\wrld:m\in M$, $\wrld:m'\in M'$:\\
if $(\wrlds:m,\wrld:m)
\mathrel{Z_{i-1}^k} (\wrlds:m',\wrld:m')$, then $(\cc{\wrlds:m}{\wrld:m},\wrld:m)
\mathrel{Z_{i}^{k+1}} (\cc{\wrlds:m'}{\wrld:m'},\wrld:m')$,
\end{itemize}
then we call $(Z^k_{i})_{\substack{i \leq \ell\\ k \in \omega}}$ an
\emph{$\hyFeat$-$(\omega,\ell)$-bisimulation} from $\str{M}$ to $\str{M}'$.
\end{definition}

The notion of $(\omega, \ell)$-bisimulation focuses on formulas of degree at most $\ell$.  
In order to obtain an $\omega$-bisimulation for countable models in terms of Areces \emph{et al.}~[1], however, it suffices to take the union of these approximating relations provided that they form a chain of inclusion as stated in the next lemma.

\begin{lemma}\label{lemma:o-bisim}
Let $\hysig$ be a hybrid signature and let $\hyFeat$ be a set of hybrid
    language features.  Let $\str{M}$ and $\str{M}'$ be countable models over
  $\hysig$.  Let $(Z^k_{i})_{\substack{i \in\omega \\ k \in \omega}}$ be a
system of relations from $\str{M}$ to $\str{M}'$ such that
\begin{enumerate}
\item $(Z^k_{i})_{\substack{i \leq \ell \\ k \in \omega}}$ is an $\hyFeat$-($\omega,\ell$)-bisimulation, and 
\item $Z^k_{0} \subseteq \dots \subseteq \mathrel{Z^k_{i}} \subseteq \dots \subseteq Z^k_{\ell}$,
\end{enumerate}
for all $\ell\in\omega$.
Then $\str{M}$
and $\str{M}'$ are $\hyFeat$-$\omega$-bisimilar.
\end{lemma}

\begin{proof}
For any $k\in\omega$, we define  $B_k=\bigcup_{\ell\in\omega} Z^k_\ell$.
It is straightforward to show that $(B_k)_{k\in\omega}$ is an $\hyFeat$-$\omega$-bisimulation.
As an example, we prove that (forth) condition holds.
Assume that $\bigl((\wrlds:m, \wrld:m), (\wrlds:n, \wrld:n)\bigl) \in B_k$ and $\wrld:m' \in M$ with $(\wrld:m, \wrld:m') \in R^\str{M}$.
By definition, there $\ell\in\omega$ such that  $(\wrlds:m, \wrld:m) \mathrel{Z_\ell^k} (\wrlds:n, \wrld:n)$. 
By (forth) property of $Z_\ell^k$, there exists $\wrld:n'\in R^\str{N}(n)$ such that $(\wrlds:m, \wrld:m') \mathrel{Z_{\ell+1}^k} (\wrlds:n, \wrld:n')$.
Hence, $\bigl((\wrlds:m, \wrld:m'), (\wrlds:n, \wrld:n')\bigl) \in B_k$.
\end{proof}

Using Lemma~\ref{karp1}, we obtain the following useful result where we
need not impose any restriction on the number of world variables: 

\begin{lemma}\label{lem:karp} 
Consider a hybrid language defined over a finite signature $\hysig$ with features from $\hyFeat$.
Let $\str{M}$ and $\str{M}'$ be Kripke structures over $\hysig$ and let $\ell \in \omega$.
Then the following are equivalent: 
\begin{enumerate}[label={(\roman*)}, widest={ii}]
  \item For each hybrid \emph{sentence} $\phi$ in the hybrid language of $\hyFeat$ with $\dg(\phi)\leq \ell$, we have
\begin{equation*}
\str{M}, \wrld:m \hymodels \phi \quad\text{if and only if}\quad \str{M}', \wrld:m' \hymodels \phi.
\end{equation*}

  \item There is an $\hyFeat$-$(\omega,\ell)$-bisimulation
$(Z^k_{i})_{\substack{i \leq \ell\\ k \in \omega}}$ from $\str{M}$ to $\str{M}'$
such that 
\begin{enumerate}
\item for any constant sequences $\wrlds:m\in M^k$ and $\wrlds:m'\in
{M'}^k$, with $\wrlds:m(j) = \wrld:m$ and $\wrlds:m'(j) = \wrld:m'$ for all $1 \leq
j \leq k$, respectively, we have $ (\wrlds:m, \wrld:m) \mathrel{Z^k_{0}} (\wrlds:m', \wrld:m')$; and
\item $Z^k_{0} \subseteq \dots \subseteq \mathrel{Z^k_{i}} \subseteq \dots
\subseteq Z^k_{\ell}$.
\end{enumerate}
\end{enumerate}
\end{lemma}
\begin{proof}
Assume (i) and, as in the proof of Lemma~\ref{karp1}, define
$(Z^k_{i})_{\substack{i \leq \ell}}$ by letting $(\wrlds:n, \wrld:n) \mathrel{Z^k_i}
(\wrlds:n', \wrld:n')$ if and only if $\htp{\ell-i}{M}(\wrlds:m, \wrld:m) =
\htp{\ell-i}{M'}(\wrlds:n, \wrld:n)$.  By Lemma~\ref{karp1}, each
$(Z^k_{i})_{\substack{i \leq \ell}}$ is a $(k,\ell)$-bisimulation, so we only
need to verify the condition
\bgroup\makeatletter\@mathmargin28pt\makeatother\begin{equation}\tag{ext}\label{eq:Z-to-Z}
(\wrlds:m, \wrld:m) \mathrel{Z_{i-1}^k} (\wrlds:m', \wrld:m') \;\implies\;
(\cc{\wrlds:m}{ \wrld:m}, \wrld:m) \mathrel{Z_{i}^{k+1}} (\cc{\wrlds:m'}{\wrld:m'}, \wrld:m').
\end{equation}\egroup

To that end, suppose 
$\htp{\ell-(i-1)}{M}(\wrlds:m, \wrld:m) = \htp{\ell-(i-1)}{M'}(\wrlds:m', \wrld:m')$ and take an arbitrary formula $\phi\in\htp{\ell-{i}}{M}(\cc{\wrlds:m}{\wrld:m}, \wrld:m)$.
Notice that $\phi\in\htp{\ell-(i-1)}{M}(\cc{\wrlds:m}{\wrld:m}, \wrld:m)$ and $\phi = \phi(\wvar:x_1,\dots,\wvar:x_k,\wvar:x_{k+1})$ with free world variables among $\wvar:x_1,\dots,\wvar:x_k,\wvar:x_{k+1}$.
For any formula $\psi$ with free world variables among $\wvar:x_1,\dots,\wvar:x_k,\wvar:x_{k+1}$, let
$\psi^-$ be the formula  $\psi(\wvar:x_{k+1}/\wvar:x_k)$. Then, for any structure
$\str{N}$, any element $\wrld:n\in N$, and any constant sequence
$\wrlds:n$ of length $k$ we have
\begin{equation}\tag{$\dag$}\label{eq:phi-minus}
\str{N},\cc{\wrlds:n}{\wrld:n},\wrld:n\hymodels \psi\iff
\str{N},\wrlds:n,\wrld:n\hymodels \psi^-
\end{equation}
hence, we obtain the following series of equivalences
\begin{align*}
  \phi\in\htp{\ell-(i-1)}{M}(\cc{\wrlds:m}{\wrld:m},\wrld:m)
  &\iff
    \phi^-\in\htp{\ell-(i-1)}{M}(\wrlds:m,\wrld:m)\\
   &\iff
     \phi^-\in\htp{\ell-(i-1)}{\str{M}'}(\wrlds:m',\wrld:m')\\
 &\iff
    \phi\in\htp{\ell-(i-1)}{{M'}}(\cc{\wrlds:m'}{\wrld:m'},\wrld:m')  
\end{align*}
of which the first and the last follow by~\eqref{eq:phi-minus} and the second
follows by assumption.
We obtain $\phi\in\htp{\ell-i}{M'}(\cc{\wrlds:m'}{\wrld:m'},\wrld:m')$.
Since $\phi$ was arbitrary chosen, we have that $\htp{\ell-i}{M}(\cc{\wrlds:m}{\wrld:m}, \wrld:m)\subseteq \htp{\ell-i}{M'}(\cc{\wrlds:m'}{\wrld:m'},\wrld:m')$.
The proof of the converse inclusion is symmetrical.
Hence, $(\cc{\wrlds:m}{ \wrld:m}, \wrld:m) \mathrel{Z_{i}^{k+1}} (\cc{\wrlds:m'}{\wrld:m'}, \wrld:m')$.
\smallskip

For (ii) $\implies$ (i) take an arbitrary formula $\phi$ with $\dg(\phi)\leq
\ell$. Let the variables occurring in $\phi$ be from the set $\{\wrld:x_1, \dots,
\wrld:x_k\}$, so that Lemma~\ref{karp1} for $k$ applies. Since any
$(\omega,\ell)$-bisimulation restricts naturally to a $(k,\ell)$-bisimulation,
there is a $(k,\ell)$-bisimulation $(Z_i^k)_{i\leq \ell}$ between $\str{M}$ and
$\str{M}'$ such that $(\wrlds:m,\wrld:m) \mathrel{Z_{0}^k} (\wrlds:m',\wrld:m')$,
and $Z^k_{0}\subseteq \dots \subseteq Z^k_{i} \subseteq \dots \subseteq
Z^k_{n}$.  Applying Lemma~\ref{karp1}(ii) we get that
\begin{equation*}
\str{M}, \wrlds:m, \wrld:m \hymodels \phi \iff
\str{M}', \wrlds:m', \wrld:m' \hymodels \phi
\end{equation*}
as desired.
\end{proof}

It will now be useful to define an analogue of the usual notion of a \emph{type}
but only for formulas that happen to be equivalent to standard translations.

\begin{definition}[Standard translation type \& elementary equivalence]\label{def:http} 
Consider a hybrid language defined over a signature $\hysig$ with features from $\hyFeat$.
Let $(\str{M},\wrld:m)$ be a pointed model over $\hysig$.
Then $\http{M}(\wrld:m)$ denotes the collection of first-order formulas equivalent to translations of hybrid sentences holding at $\wrld:m$ in $\str{M}$, i.e., 
\begin{equation*}
\http{M}(\wrld:m) = \{ \phi \in \fohylang : \stacked{\phi \leftrightarrow \ST_x(\psi)\\
  \text{for some hybrid sentence }\psi \text{ with } \str{M}, \wrld:m \hymodels \psi \}\ \text{.}}
\end{equation*}
Two pointed $\hysig$-models $(\str{M},\wrld:m)$ and $(\str{N},\wrld:n)$ are
$\mathcal{F}$-\emph{elementarily equivalent}, in symbols,
$(\str{M},\wrld:m)\equiv(\str{N},\wrld:n)$ if they cannot be distinguished by
any hybrid sentence, i.e.,
\begin{equation*} 
\str{M}, \wrld:m \hymodels \psi \text{ iff } \str{N}, \wrld:n \hymodels \psi \text{ for all hybrid sentences $\psi$ over $\hyFeat$}
\ \text{.}
\end{equation*}
\end{definition}

By Lemma ~\ref{lemma:o-bisim} and Lemma~\ref{lem:karp}, $\omega$-bisimulation is stronger than $\mathcal{F}$-elementary equivalence.
However, the following result shows that first-order sentences cannot distinguish between $\omega$-bisimulation and $\mathcal{F}$-elementary equivalence.

\begin{theorem}\label{th:lindstrom}
Consider a hybrid language defined over a finite signature $\hysig$ with features from $\hyFeat$.
If a first-order formula $\phi(x)$ over $\fohysig$ can distinguish between two $\mathcal{F}$-elementary equivalent pointed models $(\str{M}_1, \wrld:m_1)$ and $(\str{M}_2, \wrld:m_2)$ over $\hysig$, that is,
$\str{M}_1\fomodels \phi[\wrld:m_1]$ and 
$\str{M}_2\fomodels \neg\phi[\wrld:m_2]$ and 
$(\str{M}_1, \wrld:m_1)\equiv(\str{M}_2, \wrld:m_2)$, then $\phi(x)$ can distinguish $\hyFeat$-$\omega$-bisimilar Kripke structures, that is, there exist pointed models $(\str{N}_1, \wrld:n_1)$ and $(\str{N}_2, \wrld:n_2)$ such that 
\begin{enumerate*}[label=(\alph*)]
\item $\str{N}_1\fomodels \phi[\wrld:n_1]$ and $\str{N}_2\fomodels \neg\phi[\wrld:n_2]$, and
\item $\str{N}_1$ and $\str{N}_2$ are $\hyFeat$-$\omega$-bisimilar.
\end{enumerate*}
\end{theorem}  
\begin{proof}
Let $(\str{M}_1, \wrld:m_1)$ and $(\str{M}_2, \wrld:m_2)$ be two $\mathcal{F}$-elementary equivalent pointed models.
Let $\phi(x)$ be a first-order formula such that  $\str{M}_1\fomodels \phi[\wrld:m_1]$ and $\str{M}_2\fomodels \neg\phi[\wrld:m_2]$. 
Since $\hysig = (\Prop, \Nom)$ has only finitely many atoms, Lemma~\ref{lem:karp} applies.
Thus, for each $\ell<\omega$,  there is an $(\omega,\ell)$-bisimulation
$(Z^k_{i})_{\substack{i \leq \ell\\ k \in \omega}}$
from $(\str{M}_1, \wrld:m_1)$ to $(\str{M}_2, \wrld:m_2)$
such that $(\wrlds:m_1, \wrld:m_1) \mathrel{Z^k_{0}} (\wrlds:m_2, \wrld:m_2)$
where $\wrlds:m_1$ ($\wrlds:m_2$, respectively) is a constant sequence
of elements $\wrld:m_1$ ($\wrld:m_2$, respectively), and furthermore
$Z^k_{0} \subseteq \dots \subseteq Z^k_{i} \subseteq \dots \subseteq Z^k_{\ell}$. 

In the following, we define, in first-order logic, an object-level characterization of $\hyFeat$-$(\omega,\ell)$-bisimulation between $\str{M}_1$ and $\str{M}_2$.
Expand the signature $(R, \foProp, \Nom)$ by adding:
\begin{itemize}
\item two unary predicates $U_i$ ($i \in \{1,2\}$), 
\item countably many predicates $I_{\ell}^k $ ($\ell, k \in\omega$) each of arity $(k+1)\times (k+1)$, 
\item copies $s'$ of  each nominal $s\in\Nom$, and 
\item copies $P'$ of each predicate $P\in\foProp$.
\end{itemize}
To make the notation more transparent, we will
write $I_{\ell}^k $ in the form
$I_{\ell}^k\kern-.3em\begin{pmatrix}
                    x_1,\dots,x_k,x\\
                    y_1,\dots,y_k,y
                  \end{pmatrix}$.
For each  $I_{\ell}^k$ we let
$J_{\ell}^k\kern-.3em\begin{pmatrix}
                    x_1,\dots,x_k,x\\
                    y_1,\dots,y_k,y
                  \end{pmatrix}$
stand for the formula
\begin{equation*}
\bigwedge_{1\leq j\leq k} U_1(x_j)\wedge U_1(x) \wedge
\bigwedge_{1\leq j\leq k} U_2(y_j)\wedge U_2(y) \wedge
I_{\ell}^k\kern-.3em\begin{pmatrix}
                    x_1,\dots,x_k,x\\
                    y_1,\dots,y_k,y
                  \end{pmatrix}.
\end{equation*}
Further, let $\phi'$ be the result of replacing each nominal $s$ everywhere in
$\phi$ by its copy $s'$.  Call the expanded signature $\fohysigE$.
Consider the following finite sets of formulas over $\fohysigE$.

\begin{small}
\begin{itemize}[widest={($\Psi_{\text{forth}}$)}, itemsep=4pt]
\item[($\text{init}$)]   
$\exists x, y\, J_{0}^k\kern-.3em\begin{pmatrix}
  x,\dots, x,x\\ 
  y, \dots, y,y
\end{pmatrix} 
\wedge \phi(x)  \wedge \neg \phi'(y)$
\item[($\Psi_{\text{prop}}$)]
$J_{\ell}^k\kern-.3em\begin{pmatrix}
                    x_1,\dots,x_k,x\\
                    y_1,\dots,y_k,y
                  \end{pmatrix}
\rightarrow(P(x)  \leftrightarrow P'(y))$ for all predicates $P\in\foProp$ 

\item[($\Psi_{\text{wvar}}$)]
$J_{\ell}^k\kern-.3em\begin{pmatrix}
                    x_1,\dots,x_k,x\\
                    y_1,\dots,y_k,y
                  \end{pmatrix}
  \rightarrow \bigwedge_{j\leq k} (x_j=x \leftrightarrow y_j =y)$
\item[($\Psi_{\text{forth}}$)]
$J_{\ell}^k\kern-.3em\begin{pmatrix}
                    x_1,\dots,x_k, x\\
                    y_1,\dots,y_k, y\end{pmatrix} 
\wedge U_1(z_1)\wedge x R z_1 \rightarrow
\exists z_2\, J_{\ell+1}^k\kern-.3em\begin{pmatrix}
                                 x_1,\dots,x_k,z_1\\
                                 y_1,\dots,y_k,z_2
                               \end{pmatrix}
 \wedge y R z_2$
\item[($\Psi_{\text{back}}$)]
$J_{\ell}^k\kern-.3em\begin{pmatrix}
                    x_1,\dots,x_k,x\\
                    y_1,\dots,y_k,y
                  \end{pmatrix} 
\wedge U_2(z_2)\wedge y R z_2 \rightarrow
\exists z_1\, J_{\ell+1}^k\kern-.3em\begin{pmatrix}
                                 x_1,\dots,x_k,z_1\\
                                 y_1,\dots,y_k,z_2
                               \end{pmatrix}
\wedge x R z_1$
\item[($\Psi_{\text{nom}}$)]
$J_{\ell}^k\kern-.3em\begin{pmatrix}
                    x_1,\dots,x_k,x\\
                    y_1,\dots,y_k,y
                  \end{pmatrix}
\rightarrow(x=s \leftrightarrow y = s')$ for all nominals $s\in\Nom$
                  
\item[($\Psi_{\text{bind}}$)]
$J_{\ell}^k\kern-.3em\begin{pmatrix}
                    x_1,\dots,x_k,x\\
                    y_1,\dots,y_k,y
                  \end{pmatrix}
\rightarrow
\bigwedge_{j\leq k}J_{\ell}^k\kern-.3em\begin{pmatrix}
                            x_1,\dots,x_{j-1},x,x_{j+1}\dots,x_k,x\\
                            y_1,\dots, y_{j-1},y ,y_{j+1} \dots,y_k,y
                                 \end{pmatrix}$

\item[($\Psi_{\text{atv}}$)]
$J_{\ell}^k\kern-.3em\begin{pmatrix}
                    x_1,\dots,x_k,x\\
                    y_1,\dots,y_k,y
                  \end{pmatrix}
\rightarrow
\bigwedge_{j\leq k}J_{\ell}^k\kern-.3em\begin{pmatrix}
                            x_1,\dots,x_k,x_j\\
                            y_1,\dots,y_k,y_j
                                 \end{pmatrix}$

\item[($\Psi_{\text{atn}}$)]
$ J_{\ell}^k\kern-.3em\begin{pmatrix}
                    x_1,\dots,x_k,\nom:s\\
                    y_1,\dots,y_k,\nom:s'
                  \end{pmatrix}$ for all nominals $\nom:s\in\Nom$
                                 
\item[($\Psi_{\text{ex-f}}$)]
$J_{\ell}^k\kern-.3em\begin{pmatrix}
                    x_1,\dots,x_k, x\\
                    y_1,\dots,y_k, y\end{pmatrix} 
\wedge U_1(z_1) \rightarrow
\bigwedge_{j\leq k}\exists z_2\, 
J_{\ell+1}^k\kern-.3em\begin{pmatrix}
                            x_1,\dots,x_{j-1},{z_1},x_{j+1}\dots,x_k,{x}\\
                            y_1,\dots, y_{j-1},{z_2} ,y_{j+1} \dots,y_k,{y}
                                 \end{pmatrix}$

\item[($\Psi_{\text{ex-b}}$)]
$J_{\ell}^k\kern-.3em\begin{pmatrix}
                    x_1,\dots,x_k, x\\
                    y_1,\dots,y_k, y\end{pmatrix} 
\wedge U_2(z_2) \rightarrow
\bigwedge_{j\leq k}\exists z_1\, 
J_{\ell+1}^k\kern-.3em\begin{pmatrix}
                            x_1,\dots,x_{j-1},{z_1},x_{j+1}\dots,x_k,{x}\\
                            y_1,\dots, y_{j-1},{z_2} ,y_{j+1} \dots,y_k,{y}
                                 \end{pmatrix}$
 
\item[($\Psi_{\text{ext}}$)]
$J_{\ell{-1}}^k\kern-.3em\begin{pmatrix}
                    x_1,\dots,x_k,x\\
                    y_1,\dots,y_k,y
                  \end{pmatrix}
\rightarrow
J_{{\ell}}^{k+1}\kern-.3em\begin{pmatrix}
                         x_1,\dots,x_k,x,x\\
                         y_1,\dots,y_k,y,y
                       \end{pmatrix}$
\end{itemize}
\end{small}

Notice that each set of formulas $\Psi_S$ defined above is indexed by a
  condition $S$ from Definition~\ref{def:base-bi} and Definition~\ref{ext-bi}.
  We define $\Psi_\hyFeat =\text{(init)} \cup (\bigcup_{S\in
    cond(\hyFeat)}\Psi_S)$ and we let $\Phi$ be the set of universal closures of
  $\Psi_\hyFeat$.  We claim that every finite subset of $\Phi$ is consistent.
To see this let $\Phi_0 $ be one such finite subset and let $k, \ell \in \omega$
be such that $J_{i}^j$ occurs in $\Phi_0$ only if $j \leq k$ and $i \leq
\ell$. Recall that there is an $(\omega,\ell)$-bisimulation
$(Z^k_{i})_{\substack{i \leq \ell\\ k \in \omega}}$ from $(\str{M}_1,
\wrld:m_1)$ to $(\str{M}_2, \wrld:m_2)$ such that $(\wrlds:m_1, \wrld:m_1)
\mathrel{Z^k_{0}} (\wrlds:m_2, \wrld:m_2)$ where $\wrlds:m_1$ and $\wrlds:m_2$
are constant sequences of length $k$, consisting of elements $\wrld:m_1$ and
$\wrld:m_2$, respectively, and furthermore $Z^k_{0}\subseteq \dots \subseteq
Z^k_{i} \subseteq \dots \subseteq Z^k_{\ell}$.  We can suppose without loss of
generality that $M_1 \cap M_2 = \emptyset$ (if this is not the case, take
disjoint isomorphic copies of $M_1$ and $M_2$).  We construct $\str{M}_3$ by
putting:
\begin{itemize}
\item $M_3 = M_1 \cup M_2$
\item $R^{\str{M}_3} = R^{\str{M}_1} \cup R^{\str{M}_2}$
\item $U_{i}^{\str{M}_3} = M_i$ for all $i \in \{1,2\}$
\item $P^{\str{M}_3} = P^{\str{M}_1}$ and  ${P'}^{\str{M}_3}= P^{\str{M}_2}$ for all predicates $P\in\foProp$
\item $s^{\str{M}_3} = s^{\str{M}_1} $ and $s^{\prime\str{M}_3} = s^{\str{M}_2} $ for all nominals $s\in\Nom$
\item $(I_{i}^j)^{\str{M}_3} = Z_{i}^j$ for  all $j \leq k$ and $i \leq \ell$
\end{itemize}

This construction produces a model, since we have no non-constant functions in the signature.  
It follows that $\str{M}_3 \fomodels \Phi_0$, as each set of formulae $\psi_S$,
  where $S\in cond(\hyFeat)$, used to define $\Phi$ are simply restatements in
  first-order logic of conditions appearing in Definition~\ref{def:base-bi} and
  Definition~\ref{ext-bi}.

Since every finite subset of $\Phi$ is consistent, by compactness we get that
$\Phi$ has a model, say, $\str{M}_4$, and by the L\"owenheim-Skolem theorem we can also take it to be countable.  Thus, there are $\wrld:m, \wrld:n \in M_4
$ such that $\str{M}_4 \fomodels \phi[\wrld:m]$, and $\str{M}_4 \fomodels
\neg\phi'[\wrld:n] $.  
\begin{center}
\begin{tikzcd}
 & \str{M}_4 \ar[r,dotted,no head] &\fohylangE&\\
\substack{\str{M}_5 \fomodels \phi[m] \\  \str{M}_6 \fomodels \neg \phi[n]} \ar[r,dotted,no head]& \fohylang \ar[rr,"b",swap] \ar[ur,hook] && \fohylangP \ar[ul,hook] & \str{M}'_6 \fomodels\neg \phi'[n]\ar[l,dotted,no head]
\end{tikzcd}
\end{center}
Taking the reduct of $\str{M}_4$ to the signature $\fohysig$ we obtain a model $\str{M}_5 \fomodels \phi[\wrld:m]$.
Let $\fohysigP$ be the first-order signature obtained from $\fohysig$ by replacing each nominal $s\in\Nom$ with its copy $s'$ and each predicate $P\in\foProp$ with its copy $P'$.
We obtain a bijective mapping $b: \fohylang \to \fohylangP$.
Now, taking the reduct of $\str{M}_4$ to the signature $\fohysigP$ we obtain a model $\str{M}'_6\fomodels\neg \phi'[\wrld:n]$.  
Then reducing $\str{M}'_6$ across the bijection $b$, we obtain a model $\str{M}_6 \fomodels \neg\phi[\wrld:n]$.
By the countability of $\str{M}_4$, the system of relations
 $(I^k_\ell)_{\substack{i \leq \ell\\ k \in \omega}}^{\str{M}_4}$ is an $\hyFeat$-$(\omega,\ell)$-bisimulation for all $\ell\in\omega$.
 By Lemma~\ref{lemma:o-bisim}, the structures $\str{M}_5$ and $\str{M}_6$ are $\hyFeat$-$\omega$-bisimilar.
\end{proof}

\begin{theorem}\label{thm:invariance} 
Consider a hybrid language defined over a finite signature $\hysig$ with features from $\hyFeat$.
A first-order formula $\phi(x)$ is equivalent to the standard translation of a hybrid sentence if, and only if, 
it is invariant under $\hyFeat$-$\omega$-bisimulations.
\end{theorem}
\begin{proof}
The forward direction follows immediately by Lemma~\ref{lem:karp} and
property~\eqref{eq:SE}.  For the backward direction, first, we show that
\begin{equation}\tag{i}\label{eq:i}\textstyle
\Mod(\phi) = \bigcup_{(\str{M}, \wrld:m) \in \Mod(\phi)} \Mod(\http{M}(\wrld:m)).
\end{equation}     

Assume that $(\str{M}_2, m_2) \in \bigcup_{(\str{M}, \wrld:m) \in \Mod(\phi)} \Mod(\http{M}(\wrld:m))$.
It follows that $\str{M}_2 \fomodels \http{M_{\text{$1$}}}(\wrld:m_1)[\wrld:m_2]$
for some pointed model $(\str{M}_1, \wrld:m_1)$  such that
$\str{M}_1\fomodels\phi[\wrld:m_1]$. Note that for any formula $\psi$ equivalent to
the translation of a hybrid sentence, we have
$\str{M}_1  \fomodels \psi[\wrld:m_1]$ if and only if
$\str{M}_2,\fomodels\psi[\wrld:m_2]$.
Therefore, $(\str{M}_1,\wrld:m_1)$ and $(\str{M}_2,\wrld:m_2)$ are $\mathcal{F}$-elementary equivalent in their hybrid language, in symbols,
$(\str{M}_1,\wrld:m_1)\equiv(\str{M}_2,\wrld:m_2)$.
Assume for reductio that $\str{M}_2 \nvDash \phi[\wrld:m_2]$, so $\str{M}_2
\fomodels \neg \phi[\wrld:m_2]$. 
By Theorem~\ref{th:lindstrom}, there exist $(\str{N}_1, \wrld:n_1)$ and $(\str{N}_2, \wrld:n_2)$ $\hyFeat$-$\omega$-bisimilar such that 
$\str{N}_1\fomodels \phi[\wrld:n_1]$ and $\str{N}_2\fomodels \neg\phi[\wrld:n_2]$, which is a contradiction with the invariance of $\phi$ under $\hyFeat$-$\omega$-bisimulation.

Secondly, take an arbitrary $(\str{M}, \wrld:m) \in \Mod(\phi)$.  Note that~\eqref{eq:i} implies that every model of
$\http{M}(\wrld:m)$ must be a model of $\phi$. But then $\http{M}(\wrld:m) \cup \{\neg
\phi\}$ is unsatisfiable.  By compactness of first-order logic, $\Psi_{(\str{M},
  \wrld:m)} \cup \{\neg \phi\}$ is unsatisfiable for some finite $\Psi_{(\str{M},
  \wrld:m)}\subseteq \http{M}(m)$ (and we can pick a unique one for each
$(\str{M},\wrld:m)$, if necessary using the axiom of choice). Hence, $\bigwedge
\Psi_{(\str{M}, \wrld:m)} $ logically implies $\phi$. Then using~\eqref{eq:i} we
obtain,
\begin{equation}\tag{ii}\label{eq:ii}\textstyle
\Mod(\phi) = \bigcup_{(\str{M}, \wrld:m) \in
  \Mod(\phi)} \Mod(\bigwedge \Psi_{(\str{M}, \wrld:m)})
\ \text{.} 
\end{equation}
However, this means that the set $\{\phi\} \cup \bigl\{ \neg \bigwedge
\Psi_{(\str{M}, \wrld:m)}: (\str{M},m) \in \Mod(\phi)\bigr\}$ is unsatisfiable, and by
compactness again, for some finite $\Theta \subseteq \{\neg \bigwedge
\Psi_{(\str{M}, \wrld:m)} : (\str{M}, \wrld:m) \in \Mod(\phi)\}$, the set $\{ \phi \} \cup
\Theta$ is unsatisfiable. Consequently, $\Mod(\phi) \subseteq \Mod(\neg
\bigwedge \Theta)$.  By its definition $\Theta$ is logically equivalent to
\begin{equation*}\textstyle
\neg
\bigwedge \Psi_{(\str{M}_1, \wrld:m_1)}\wedge \neg \bigwedge \Psi_{(\str{M}_2, \wrld:m_2)}
\wedge\dots\wedge \neg \bigwedge \Psi_{(\str{M}_k, \wrld:m_k)}
\end{equation*}
for some $k$, and
therefore we have that $\neg\bigwedge\Theta$ is equivalent to
\begin{equation*}\textstyle
\bigwedge
\Psi_{(\str{M}_1, \wrld:m_1)}\vee \bigwedge \Psi_{(\str{M}_2, \wrld:m_2)} \vee\dots\vee
\bigwedge \Psi_{(\str{M}_k, \wrld:m_k)}
\ \text{.}
\end{equation*}
Hence
\begin{equation*}\textstyle
\Mod(\neg \bigwedge \Theta) \subseteq 
\bigcup_{(\str{M}, \wrld:m) \in \Mod(\phi)} \Mod(\bigwedge \Psi_{(\str{M}, \wrld:m)})
\end{equation*}
so, using~\eqref{eq:ii} we get, 
\begin{equation}\tag{iii}\label{eq:iii}\textstyle
\Mod(\phi) = \Mod(\neg \bigwedge \Theta).
\end{equation}
But now note that $\neg \bigwedge \Theta$ is equivalent to a translation of a
hybrid sentence, since the translations are closed under finite conjunctions and
disjunctions.
\end{proof}

The next result follows easily by the methods we have developed.

\begin{theorem}\label{thm:axiomatisation}
Consider a hybrid language defined over a finite signature $\hysig$ with features from $\hyFeat$.
 Then, $K$ is axiomatisable by a hybrid sentence $\phi$ over $\fohysig$ if and only if for some some $\ell$, the class $K$ is closed under $\hyFeat$-$(\omega,\ell)$-bisimulations. 
\end{theorem}
\begin{proof}
Let $\phi$ be a hybrid sentence axiomatising $K$, and let $\ell$ be the degree of $\phi$. 
Then the left-to-right direction follows by Lemma~\ref{lem:karp}.  
For the converse direction, assume $K$ is closed under
$\hyFeat$-$(\omega,\ell)$-bisimulations. Since the language has finitely
many atoms there are only finitely many formulas of degree at most $\ell$
up to logical equivalence.  Let $\htp{\ell}{M}(\wrld:m)$ denote the set of sentences of degree at most $\ell$ satisfied by $(\str{M},\wrld:m)$. 
Then for any model $(\str{M}, \wrld:m) \in K$, the set $\htp{\ell}{M}(\wrld:m)$ is finite.  Hence, $\bigvee_{(\str{M}, \wrld:m) \in K} \bigwedge\htp{\ell}{M}(\wrld:m)$ can be seen as a finitary disjunction up to
logical equivalence, and we may call it $\phi$.   
Then for any pointed model $\str{M}, \wrld:m \hymodels \phi$, we must have that $(\str{M}, \wrld:m) \in K$ by Lemma~\ref{lem:karp} again, as $K$ is closed under $\hyFeat$-$(\omega,\ell)$-bisimulations.
\end{proof}

\section{Undecidability of invariance}

In van Benthem~\cite{vB96} it is shown that invariance under standard
bisimulations is undecidable for first-order formulas.  Hodkinson and
Tahiri~\cite{HT10}, using the same proof, lift van Benthem's result to
quasi-injective bisimulations and versions thereof. Indeed, the same proof shows
undecidability of invariance under about every nontrivial relation between
models, as we will now demonstrate.  To make the argument applicable to
arbitrary signatures, we will now define a version of disjoint union of models
applicable to arbitrary signatures.

\begin{definition}\label{def:disjoint}
Let $\str{A}$ and $\str{B}$ be models of some signature $\fosig$.
We define $\str{A}\uplus\str{B}$ to be the model whose universe
is $A\uplus B$, the interpretation of a relation $R$ is
$R^\str{A}\uplus R^\str{B}$, the interpretation of a $k$-ary function
$f(x_1,\dots, x_k)$ is defined by
\begin{equation*}
f^{\str{A}\uplus\str{B}}(u_1,\dots,u_k) =
\begin{cases}
f^{\str{A}}(u_1,\dots,u_k), & \text{if } (u_1,\dots,u_k)\in A^k\\
f^{\str{B}}(u_1,\dots,u_k), & \text{if } (u_1,\dots,u_k)\in B^k\\
u_i, & \text{otherwise with least $i$ s.\,t.\ $u_i \in A$}
\end{cases}  
\end{equation*}
and finally the interpretation of a constant $c$ is $c^\str{A}$.
\end{definition}  

Note that under this definition we have that $\str{A}\fomodels\phi$ if and only if
$\str{A}\uplus\str{B}\fomodels \phi^A$, where $\phi^A$ is the relativisation of
$\phi$ to $A$.

\begin{definition}\label{def:genuine}
Let $\fosig$ be a first-order signature. A relation $S$ between $\fosig$-models
will be called \emph{genuine} if it satisfies the following three conditions:
\begin{enumerate}[label={(\alph*)}, widest={b}]
\item There exist a formula $\varphi_S(x)$ with $x$ free, and models
$\str{A}$ and $\str{B}$ such that
$\str{A}$ and $\str{B}$ are $S$ related, but
for some $a\in A$ and some $b\in B$ we have $\str{A}\fomodels\varphi_S(a)$  and 
$\str{B}\not\fomodels\varphi_S(b)$ (nontriviality).
\item If $\str{A}$ and $\str{B}$ are $S$ related, then
  for any model $\str{C}$, the models
$\str{A}\uplus\str{C}$ and $\str{B}\uplus\str{C}$ are $S$
related (preservation under common disjoint extensions).
\item If $\str{A}$ and $\str{B}$ are $S$ related, and
$\str{A}'$ and $\str{B}'$ are expansions of $\str{A}$ and
$\str{B}$ to a signature $\fosig' \supseteq \fosig$, then
$\str{A}'$ and $\str{B}'$ are $S$ related (stability under language
expansions).
\end{enumerate}  
\end{definition}

\begin{lemma}\label{lem:undec}
For any first-order signature $\fosig$ and any genuine relation $S$,
it is undecidable whether a first-order sentence $\sigma$ is invariant under $S$.
\end{lemma}  
\begin{proof}
Expand $\fosig$ by adding a unary predicate $P$ and a constant $d$,
and for any $\fosig$-sentence
$\sigma$, let $\psi_\sigma$ be the formula
$\varphi_S(x)\vee (\exists x\, P(x)\rightarrow\sigma^P)$, where
$\sigma^P$ is the relativisation of $\sigma$ to $P$. We will show that
$\sigma$ is valid if and only if $\psi_\sigma$ is invariant under $S$, thereby
reducing invariance under $S$ to validity.

If $\sigma$ is valid, then so is $\psi_\sigma$, hence $\psi_\sigma$ is invariant
under $S$. If $\sigma$ is not valid, let $\str{C}$ be a model with
$\str{C}\fomodels\neg\sigma$. Since $S$ is genuine, we have $S$-related models
$\str{A}$ and $\str{B}$ such that $\str{A}\fomodels\varphi_S(a)$ and
$\str{B}\not\fomodels\varphi_S(b)$ for some $a\in A$ and some $b\in B$.  Let
$\str{N}$ be $\str{A}\uplus\str{C}$, and let $\str{M}$ be
$\str{B}\uplus\str{C}$, where $P^\str{N} = C$, $P^\str{M} = C$, $d^\str{N} = a$,
and $d^\str{M} = b$. Then, we have $\str{N}\fomodels\phi_\sigma(d^\str{N})$, and
$\str{M}\not\fomodels\phi_\sigma(d^\str{M})$. But $\str{N}$ and $\str{M}$ are
$S$-related, showing that $\psi_\sigma$ is not preserved under $S$.
\end{proof}

It should be clear that $\hyFeat$-$\omega$-bisimilarity, for any hybrid language built with features from $\hyFeat$, is a genuine relation in the sense of Definition~\ref{def:genuine}. The result below is then immediate.

\begin{corollary}
Invariance under $\hyFeat$-$\omega$-bisimulations is undecidable, for any hybrid language built with $\hyFeat$.
\end{corollary}  

\section{Final remarks}\label{sec:final}
Bisimulations in general, and the ones considered by us in particular, are equivalence relations on classes of models. 
Any reasonable notion of bisimulation is weaker than isomorphism, of course. 
Weaker conditions imposed on a bisimulation are easier to satisfy, so more models are bisimilar, hence fewer formulas are invariant under weak bisimulations. 
Conversely, the more we demand of a bisimulation, the \emph{more} formulas will be invariant.  
The threshold is reached when we demand enough to get a concept at least as strong as  elementary equivalence, since then \emph{all} formulas will be invariant. 
In our case,  for $\hyFeat \supseteq\{\exists,\jump\}$ the notion of $\hyFeat$-$\omega$-bisimulation is strictly stronger than
elementary equivalence; hence, for such an $\hyFeat$, invariance under $\hyFeat$-$\omega$-bisimulations is trivial.

\begin{center}\vspace*{.8ex plus 10pt}
\begin{tabular}{ll}
\toprule
Hybrid features & Invariance under \\
\midrule
$\emptyset$ & Bisimulations~=~$\emptyset$-$\omega$-bisimulations \\
  $\{\store\}$ w/o nominals & $\{\store\}$-$\omega$-bisimulations\\
  $\{\store\}$ with nominals & $\{\store,\Nom\}$-$\omega$-bisimulations\\
  $\{\jump\}$ & $\{\jump\}$-$\omega$-bisimulations\\
$\{\jump,\store\}$ & $\{\jump,\store\}$-$\omega$-bisimulations\\  
$\{\exists\}$ & $\{\exists\}$-$\omega$-bisimulations  \\
$\{\jump,\exists\}$ & equivalent to FOL \\  
full feature set & equivalent to FOL \\
\bottomrule
\end{tabular}\vspace*{.8ex plus 10pt}
\end{center}
We end the article with two questions. The first one has to do with the
signature $\{\exists\}$ (or, equivalently, $\{\store,\exists\}$).  As we
mentioned a number of times, it is well known that hybrid logic in the full
signature is equivalent to first-order logic, and $\jump$ is the operator that
mimics identity, it seems of interest to compare $\{\exists\}$-bisimulations
with existing equivalence relations on models for first-order logic without
identity, such as the ones considered in Casanovas~\emph{et al.}~\cite{CDJ96}.

\begin{question}
How do $\{\exists\}$-$k$-bisimulations and $\{\exists\}$-$\omega$-bisimulations
compare to various equivalence relations on classes of models for logic without
equality? 
\end{question}  

Another interesting direction is to consider bisimulation invariance in the
finite, more precisely, whether the bisimulation characterisation
theorems presented here (or other similar theorems from the literature for
hybrid logic) still hold over finite Kripke structures. Rosen~\cite{Ros97} (see
also Otto~\cite{Ott06}) showed that the usual characterisation for modal logic
is preserved in moving to a finite context. Unfortunately, the only proofs of
this fact the authors are aware of depend heavily on the unravelling
construction, which as we saw is not always available in the hybrid context. The
connection to games seems more promising, but we are not inclined to state any
conjectures here.

\begin{question}
Which of the preservation-under-bisimulations theorems mentioned here
still hold on finite structures?   
\end{question}



\bibliographystyle{asl}
\bibliography{hl-rsl}

\end{document}